\documentclass[12pt,reqno]{amsart}
\usepackage[T1]{fontenc}
\usepackage{amsfonts}
\usepackage{amssymb}
\usepackage{mathrsfs}
\usepackage[latin1]{inputenc}
\usepackage{amsmath}
\usepackage[english]{babel}
\usepackage{setspace}
\usepackage{bbm}

 \newtheorem{thm}{Theorem}[section]
 
 \newtheorem{lem}[thm]{Lemma}
 \newtheorem{prop}[thm]{Proposition}
 \theoremstyle{definition}

 \newtheorem{rem}[thm]{Remark}
 \newtheorem{ex}[thm]{Example}  %
 \numberwithin{equation}{section}

 \numberwithin{equation}{section}

\newcommand{\R}{{\mathbb R}}
\newcommand{\D}{{\mathbb D}}

\newcommand{\C}{{\mathbb C}}
\newcommand{\N}{{\mathbb N}}

\newcommand{\cL}{{\mathcal L}}

\newcommand{\cs}{\mathsf{s}}

\newcommand\Ker{\mathop{\rm Ker}}

\begin{document}

%
%
%
%

\title[Spectral and ergodic properties of the operator $B(r,s)$ ]
 {Spectral and ergodic properties of the operator $B(r,s)$ over power series spaces $\Lambda_\infty(\alpha)$ of infinite type and  their duals}

\author[A.A. Albanese, C. Mele]{Angela A. Albanese, Claudio Mele}

\address{Angela A. Albanese\\
	Dipartimento di Matematica e Fisica ``E. De Giorgi''\\
	Universit\`a del Salento- C.P.193\\
	I-73100 Lecce, Italy}
\email{angela.albanese@unisalento.it}

\address{Claudio Mele\\
	Dipartimento di Matematica e Fisica ``E. De Giorgi''\\
	Universit\`a del Salento- C.P.193\\
	I-73100 Lecce, Italy}
\email{claudio.mele1@unisalento.it}


\begin{abstract}
	In this paper, we investigate the spectral and ergodic properties of the linear operator $B(r,s)$ acting on power series spaces $\Lambda_\infty(\alpha)$ of infinite type and on their strong duals.  Precisely, we provide a complete characterization of its fine spectrum and establish necessary and sufficient conditions for the operator to be power bounded and (uniformly) mean ergodic.
\end{abstract}

\maketitle
\section{Introduction}
The spectral problem of infinite matrices has attracted significant interest in mathematics as well as in applied sciences, leading to  an increasing research in this setting over the years.
This is because spectral properties  provide a useful tool for analyzing operators associated with various problems. For example, infinite systems of linear equations can be represented using infinite matrices with appropriate coefficients. Another important application is the study of eigenvalues of integral operators and their corresponding infinite matrices.

An extensive research has been carried out in this field, with contributions from numerous mathematicians. 
In particular, 
the spectral properties of the  operator $B(r, s)$ defined by  the double band matrix
$$
B(r,s)=\left[\begin{matrix}r&0&0&\dots \\ s&r&0&\dots \\ 0&s&r&\dots \\ \vdots&\vdots&\vdots&\ddots\end{matrix}\right],
$$
with $r,s\in\R$, $r,s \neq 0$, i.e., 
$(B(r, s)x)_n = sx_{n-1}+ rx_n$, for  $n\in\N_0$, 
with $x_{-1} := 0$, were studied by several authors.  We note that the
operator $B(r, s)$ is reduced to the right-shift, difference and Zweier matrices in special
cases $(r, s)$ = $(0, 1)$, $(r, s) = (1,-1)$ and $(r, s) = (r, 1 - r)$, respectively. We mention that  the fine spectrum of the operator $B(r,s)$ was determined by Altay and Ba\c{s}ar  \cite{AB}, Furkan, Bilgi\c{c} and Kayaduman \cite{F}, Bilgi\c{c} and
Furkan \cite{BF2} over the sequence spaces $c_0$ and $c$, $l_1$ and $bv$, $l_p$ and $bv_p$, for $1 < p <\infty$,
respectively. The case $(r,s)=(1,-1)$ was studied by Altay and Ba\c{s}ar  \cite{ALBA}, Kayaduman
and Furkan \cite{KF}, Akhmedov and Ba\c{s}ar \cite{AKB,AKB2} over the sequence spaces $c_0$ and $c$, $l_1$ and
$bv$,  $l_p$ ($1\leq p<\infty$), $bv_p$ ($1\leq p<\infty$), respectively. The operator
$B(r, s)$ was  generalized to the triple
band operator $B(r, s, t)$ and its spectral properties were also widely studied. In particular,   Bilgi\c{c} and Furkan  \cite{BF1} examined
the  fine spectrum of $B(r, s, t)$  over the sequence spaces $l_1$ and $bv$. Further studies by Altay, Bilgi\c{c} and Furkan \cite{ABF} investigated its fine spectrum over $c_0$ and $c$,
while Ba\c{s}ar, Bilgi\c{c} and Furkan \cite{BBF} determined its spectrum over $l_p$ and $bv_p$, with $1<p<\infty$.
Recently, Saad R. El-Shabrawy \cite{Sh-1} studied the spectral properties of $B(r,s,t)$ on the Hahn sequence space $h$ and the space $bv_0$, which consists of null sequences of bounded variation. For further results and generalizations of the operator $B(r,s)$  we refer to \cite{AKEL-1,AKEL,AKEL-2,Sh,Sh-3}.

The aim of this paper is to describe the spectral and ergodic properties of the operator $B(r,s)$ acting on power series spaces $\Lambda_\infty(\alpha)$ of infinite type, as well as on their duals. Power series spaces of infinite type form an important class of K\"othe spaces. For instance, the space of holomorphic functions on $\C$ and the space $\cs$ of rapidly decreasing sequences are examples of power  series spaces of infinite type. To this end, we first examine the operator $B(r,s)$ when it acts on  weighted sequence Banach spaces, a setting not previously explored in the literature. In particular, we investigate the fine spectrum of $B(r,s)$ acting on $l_p(v)$ for $1 < p < \infty$, with a given weight $v$, providing a complete characterization. Moreover, we establish necessary and sufficient conditions under which the operator $B(r,s)$ is power bounded and (uniformly) mean ergodic in $l_p(v)$. Subsequently, we extend our analysis to the power series spaces $\Lambda_\infty(\alpha)$ of infinite type and their duals $\Lambda_\infty(\alpha)'$. We show that $B(r,s)$ acts continuously on both spaces $\Lambda_\infty(\alpha)$ and $\Lambda_\infty(\alpha)'$, provided a suitable condition on the sequence $\alpha$ is satisfied. Finally, we describe the spectral and ergodic properties of $B(r,s)$ over these non-normable sequence spaces.

The paper is organized as follows. Section 2 introduces fundamental concepts of spectral theory, including the fine spectrum of linear operators in Banach spaces, along with basic notions of dynamics. In Section 3, we recall the main results concerning $B(r,s)$ in the classical Banach spaces $l_p$ from \cite{BF2}, and then study its spectral and ergodic properties in the weighted sequence spaces $l_p(v)$. Section 4 begins with a presentation of the Fr\'echet spaces $\Lambda_\infty(\alpha)$ and their strong duals $\Lambda_\infty(\alpha)'$, along with relevant notation and results pertaining to the non-normable setting. We then investigate the spectral and ergodic properties of the operator $B(r,s)$ in these spaces.


\section{Definitions and general results}
Let $X$ be a non trivial complex Banach space. We  denote by $\cL(X)$ the set of all bounded (i.e., continuous) linear operators on $X$ into itself. Equipped with the topology of pointwise convergence $\tau_s$  on $X$ (i.e., the strong operator topology), the space $\cL(X)$ is denoted by $\cL_s (X)$ and with the topology $\tau_b$ of uniform convergence on the closed unit ball $B_X$ of $X$ (i.e., on bounded sets of $X$), the space $\cL(X)$ is denoted by $\cL_b(X)$.

 For $T\in \cL(X)$,  the adjoint $T^*$ of $T$ is the continuous linear operator on the strong dual $X^*$ of $X$, defined by $(T^*\phi)(x) := \phi(Tx)$ for all $\phi \in X^*$ and $x \in X$. Clearly, $T^*\in \cL(X^*)$. 
We denote by $R(T)$ the range of $T$, i.e.,
$R(T) := \{y \in X \colon y = Tx,\; x \in X\}.$

	Given a complex Banach space  $X$ and $T\in \cL(X)$, the resolvent set $\rho(T,X)$ of $T$ consists of all $\alpha\in\C$ such that $R(\lambda,T):=(T-\alpha I)^{-1}$ exists in $\cL(X)$. The set $\sigma(T,X):=\C\setminus \rho(T,X)$ is called the \textit{spectrum} of $T$. The set $\sigma(T,X)$ is decomposed into three mutually disjoint parts, which constitute the \textit{fine spectrum} of $T$. Namely, we have 
the \textit{point spectrum}
\[
\sigma_{p}(T,X):=\{\alpha\in\C\,\colon \Ker(T-\alpha I)\not=\{0\} \}
\]
of $T$, the \textit{continuous spectrum}
\begin{align*}
	& \sigma_c(T,X):=\\
	&\{\alpha\in\C\,\colon \Ker(T-\alpha I)=\{0\} \mbox{ and}\ (T-\alpha I)(X)\not=X\  \mbox{with}\ \overline{(T-\alpha I)(X)}=X\}
\end{align*}
of $T$, and the \textit{residual spectrum}
\[
\sigma_r(T,X):=\{\alpha\in\C\,\colon \Ker(T-\alpha I)=\{0\} \mbox{ and}\ \overline{(T-\alpha I)(X)} \not=X\}
\]
of $T$. A point $\lambda\in \sigma_{p}(T,X)$ is called an \textit{eigenvalue}  of $T$. Simple arguments of duality show that   $\sigma(T,X)=\sigma(T^*,X^*)$ and that
$	\sigma_{p}(T,X)\subseteq \sigma_{p}(T^*,X^*)\cup\sigma_{r}(T^*,X^*)$,  $\sigma_r(T,X)\subseteq \sigma_{p}(T^*,X^*)$.
More precisely, we have
$\{\lambda\in\sigma_{p}(T,X)\,\colon \overline{(\lambda I-T)(X)}\not=X\}\subseteq \sigma_{p}(T^*,X^*)$
and
$	\{\lambda\in\sigma_{p}(T,X)\,\colon \overline{(\lambda I-T)(X)}=X\}\subseteq \sigma_r(T^*,X^*)$.

We denote by $\omega$ the space of all complex valued sequences $(x_n)_{n\in\N_0}\in\C^{\N_0}$, where $\N_0:=\{0,1,\ldots, m,\ldots\}$. Note that $\omega$, equipped with the locally convex topology of  coordinatewise convergence, is a Fr\'echet space. A normed space $\lambda$ is called a \textit{sequence space} if $\lambda\subseteq \omega$ with a continuous inclusion.

We also recall that 
an operator $T\in \cL(X)$, with $X$ a Banach space,  is called \textit{power bounded} if $\{T^n\}_{n\in\N_0}$ is an equicontinuous subset of $\cL(X)$ (i.e., $\sup_{n\in\N_0}\|T^n\|<\infty$). The  Ces\`aro means of an operator $T\in\cL(X)$ are defined by
\[
T_{[n]} :=\frac{1}{n}\sum_{m=1}^nT^m,\quad n\in\N.
\]
The operator  $T$ is called \textit{mean ergodic} (resp. \textit{uniformly mean ergodic}) if $\{T_{[n]}\}_{n\in\N}$ 
is a convergent
sequence in $\cL_s(X)$ (resp. in $\cL_b(X)$). 
The  Ces\`aro means of $T$ satisfy the following identities
\begin{equation}\label{cesarof}
	\frac{T^n}{n}= T_{[n]}-\frac{n-1}{n} T_{[n-1]}, \quad  n \geq 2.
\end{equation}
So, it is clear that $\frac{T^n}{n}\to 0$ in $\cL_s(X)$ as $n\to\infty$, whenever $T$ is mean ergodic.
Furthermore, if $T\in\cL(X)$ is mean ergodic, the operator $P:=\lim_{n\to\infty}T_{[n]}$ in $\cL_s(X)$ is a projection on $X$ satisfying ${\rm Im}\, P = \Ker(I- T)$ and $\Ker P =\overline{{\rm Im}\,(I-T)}$, with
\[
X = \overline{{\rm Im}\,(I-T)}\oplus \Ker(I-T),
\] 
(see \cite[Ch. VIII; \S 3]{Y}).  In reflexive complex Banach spaces $X$ every power bounded operator is necessarily mean ergodic  (see \cite[Corollary 4, \S VIII.5]{DS}). 

An operator $T\in\cL(X)$, with $X$ a Banach space,  is said to be \textit{hypercyclic} if there exists $x\in X$ whose orbit under $T$ is dense in $X$, that is the set ${\rm orb}(x, T) := \{T^nx \colon  n\in\N_0\}$ is dense in $X$. The operator $T$ is said to be  \textit{supercyclic} if  there exists $x\in X$ whose projective orbit under $T$ is dense in $X$, that is the set $\{\lambda T^nx \colon  n\in\N_0, \ \lambda\in\C\}$ is dense in $X$.

No power-bounded operator is hypercyclic, but can be supercyclic. Every hypercyclic operator is always supercyclic. The converse is not true in general.

\section{The operator $B(r, s)$ on sequence spaces}
\subsection{$B(r,s)$ acting on the Banach spaces $l_p$, $1<p<\infty$.}
Let us consider the generalized difference operator $B(r, s)$ represented by the band matrix with $r, s \in \R$, $r,s \neq 0$, defined by
$$B(r,s)=\left[\begin{matrix}r&0&0&\dots \\ s&r&0&\dots \\ 0&s&r&\dots \\ \vdots&\vdots&\vdots&\ddots\end{matrix}\right].$$
This means that $B(r,s)x:=(sx_{n-1}+rx_n)_{n\in\N_0}$, where $x\in\omega$ with $x_{-1}:=0$. 

In this section, we first recall the fine spectrum classification of the operator $B(r,s)$ acting on the sequence space $l_p$, for $1<p<\infty$, given in \cite{BF2}.  
 We begin with an estimate of the operator norm of the operator $B(r , s )$, contained in \cite[Theorem 2.1]{BF2}.
 
\begin{prop}\label{ineqb} Let $r,s\in\R$ and $r,s\neq 0$,
$1 < p < \infty$. Then $B(r , s)\colon l_p \to l_p$ is a bounded linear operator satisfying the inequalities
$$(|r|^p + |s|^p)^{\frac{1}{p}} \leq \|B(r, s)\|\leq |r|+|s|.$$
\end{prop}

The complete characterization of the spectrum of $B(r , s )$, acting on the sequence space $l_p$, is provided in \cite[Theorem 2.3]{BF2}. The point, continuous and residual spectrum of $B(r , s )$, acting on the sequence space $l_p$, have been described in  \cite[Theorems 2.4, 3.8, 2.11]{BF2}. If we set $\D_r(z):=\{w\in\C\colon |w-z|<r\}$ for $z\in\C$ and $r>0$ (for $z=0$ we simply write $\D_r$ instead of  $\D_r(0)$ and for $z=0$ and $r=1$ we simply write $\D$ instead of  $\D_1(0)$), the fine description of the spectrum of $B(r , s )$, acting on the sequence space $l_p$, is contained in the following theorem.

\begin{thm}\label{spbrs} Let $r,s\in\R$ and $r,s\neq 0$, 
	$1 < p < \infty$. 
\begin{enumerate}
\item $\sigma(B(r , s), l_p)=\left\{\alpha\in\C\colon |r-\alpha|\leq |s|\right\}=\overline{\D}_{|s|}(r)$;
\item $\sigma_p(B(r , s), l_p)=\emptyset;$
\item $\sigma_r(B(r , s,), l_p)=\sigma_p(B^*(r , s), l^*_p)=\left\{\alpha\in\C\colon |r-\alpha|<  |s|\right\}=\D_{|s|}(r);$
\item $\sigma_c(B(r , s), l_p)=\left\{\alpha\in\C\colon |r-\alpha|= |s|\right\}=\partial \D_{|s|}(r).$
\end{enumerate}
\end{thm}

\begin{rem}
Since $\sigma_p(B^*(r , s ), l^*_p)=\D_{|s|}(r)$, we get that $B(r , s )$ is not supercyclic in $l^*_p$ (cf. \cite[Proposition 1.26]{BM}).
\end{rem}

We now pass to study the (uniform) mean ergodicity of the operator $B(r , s )$, acting on the space $l_p$. To this end, we observe the following  fact.

\begin{lem} \label{incspe}Let $r,s\in\R$ and $r,s\neq 0$, $L>0$ and 
	$1 < p < \infty$. Then $\overline{\D}_{L|s|}(r)\subseteq \overline{\D}$ if, and only if, $|r|+L|s|\leq 1$.
\end{lem}
\begin{proof}
 If $\overline{\D}_{L|s|}(r)\subseteq \overline{\D}$, then we have $-1\leq r-L|s|\leq r\leq r+L|s|\leq 1$. So, it follows that $|r|+L|s|\leq 1$. 

Conversely, suppose that $|r|+L|s|\leq 1$. If $\alpha\in \overline{\D}_{L|s|}(r)$, then $|r-\alpha|\leq L|s|$. Hence, $|\alpha|\leq |\alpha-r|+|r|\leq L|s|+|r|\leq 1.$ Accordingly, $\alpha\in\overline{\D}$.
\end{proof}

\begin{prop}\label{me}
Let $r,s\in\R$ and $r,s\neq 0$,  $1<p<\infty$. The following statements are equivalent:
\begin{enumerate}
\item $\frac{B^n(r,s)}{n}\to0$ in $\cL_s(l_p)$ as $n\to\infty$;
\item $B(r,s)$ is mean ergodic in $\cL(l_p)$; 
\item $B(r,s)$ is power bounded in $\cL(l_p)$;
\item $|r|+|s|\leq 1$.
\end{enumerate}
\end{prop}
\begin{proof}
(1)$\Rightarrow$(4): Since $\frac{B^n(r,s)}{n}\to\infty$ in $\cL_s(l_p)$ as $n\to\infty$, by 	\cite[Lemma 1, \S VIII.8, page 709]{DS} (cf. also \cite[Proposition 5.1]{AM}), we have that $\sigma(B(r , s ), l_p)\subseteq \overline{\D}$. In view of Theorem \ref{spbrs}(1) and Lemma \ref{incspe}, it follows that $|r|+|s|\leq 1$. 

(4)$\Rightarrow$(3): If $|r|+|s|\leq 1$, due to Proposition \ref{ineqb} we have $\|B(r,s)\|\leq 1$ and hence, $\|B^n(r,s)\|\leq 1$ for all $n\in\N_0$. Accordingly, the operator $B(r,s)$ is power bounded.

(3)$\Rightarrow$(2): Since $l_p$ is a reflexive Banach space, every power bounded operator is mean ergodic, \cite[Corollary 4, \S VIII.5]{DS} (see also \cite[Corollary 2.7, Remark 2.8]{ABR-0}).

(2)$\Rightarrow$(1): It follows by \eqref{cesarof}.
\end{proof}


\begin{prop}\label{ume}
Let $r,s\in\R$ and $r,s\neq 0$,  $1<p<\infty$. Consider the following statements:
\begin{enumerate}
\item $|r|+|s|< 1$;
\item $B^n(r,s)\to0$ in $\cL_b(l_p)$  as $n\to\infty$;
\item $B(r,s)$ is uniformly mean ergodic in $\cL(l_p)$;
\item $1\not\in\sigma(B(r,s),l_p)$.
\end{enumerate}
Then $ (1)\Rightarrow (2)\Rightarrow (3)\Rightarrow (4)$.
\end{prop}
\begin{proof}
(1)$\Rightarrow$(2): Due to Proposition \ref{ineqb}, we have that $\|B(r,s)\|\leq |r|+|s|<1$.  Clearly, it follows that $B^n(r,s)\to0$ in $\cL_b(l_p)$ as $n\to\infty$. 

(2)$\Rightarrow$(3): If $B^n(r,s)\to0$ in $\cL_b(l_p)$ as $n\to\infty$, then $B_{[n]}(r,s)\to0$ in $\cL_b(l_p)$ as $n\to\infty$. This means that  $B(r,s)$ is uniformly mean ergodic.

(3)$\Rightarrow$(4):
The uniform mean ergodicity  implies that $(I-B(r,p))(l_p)$ is a closed subspace of $l_p$  and 
\begin{equation}\label{eq.somma}
(I-B(r,s))(l_p)\oplus \Ker(I-B(r,s))=l_p,
\end{equation}
\cite{Lin} (see also \cite[Theorem 5.4]{ABR-0}). Since $\sigma_p(B(r , s), l_p)=\emptyset$, we have that $\Ker(I-B(r,s))=\{0\}$. In view of  \eqref{eq.somma}, it follows that $(I-B(r,s))(l_p)=l_p$ and hence, $(I-B(r,s))\colon l_p\to l_p$ is bijective. Accordingly,  $1\notin \sigma(B(r,s), l_p)=\overline{\D}_{|s|}(r)$. 
\end{proof}

\subsection{$B(r,s)$ acting on the Banach spaces $l_p(v)$, $1 < p < \infty$}\label{lpv}

In this part, we study the fine spectrum classification of the operator $B(r,s)$ acting on the sequence space $l_p(v)$, where $1 < p < \infty$ and $v$ is a weight sequence.  

Given  a  weight  sequence $v$, i.e., a positive sequence $v=(v_n)_{n\in\N_0}\in\omega$ and $1\leq p\leq \infty$, we define as usual
\begin{align*}
l_p(v):=\left\{x=(x_n)_{n\in\N_0}\in \omega \colon  \|x\|_{p,v}:=\|(x_nv_n)_{n\in\N_0}\|_{p}<\infty \right\},
\end{align*}
where $\|\cdot\|_{p}$ denotes the usual $l_p$ norm. 
Clearly, $(l_p(v),\|\cdot\|_{p,v})$, $1\leq p\leq \infty$,
are Banach spaces.

We begin showing a necessary and sufficient condition for $B(r,s)$ to be a continuous linear operator on $l_p(v)$.

\begin{prop}\label{contbrs}
Let $1<p<\infty$ and $v$ be a weight sequence. For each $r,s\in\R$ and $r,s\neq 0$, the operator  $B(r,s)	\in \cL(l_p(v))$ if, and only if, $\{\frac{v_{n+1}}{v_{n}}\}_{n\in\N_0}\in l_\infty$.
\end{prop}

\begin{proof}
We first suppose  that $\{\frac{v_{n+1}}{v_{n}}\}_{n\in\N_0}\in l_\infty$. Given $x\in l_p(v)$, we have that
\begin{align*}
\|B(r,s)x\|_{p,v}&=\left(\sum_{n=0}^\infty |s x_{n-1}+rx_n|^p v_n^p \right)^{\frac{1}{p}}\leq \left(\sum_{n=0}^\infty |s x_{n}|^p v_{n+1}^p \right)^{\frac{1}{p}}+\left(\sum_{n=0}^\infty |rx_n|^p v_n^p \right)^{\frac{1}{p}}\\&= |r|\|x\|_{p,v}+|s|\left(\sum_{n=0}^\infty |x_{n}|^p \frac{v_{n}^p v_{n+1}^p}{v_{n}^p} \right)^{\frac{1}{p}}\\
&\leq |r|\|x\|_{p,v}+|s|\left\| \left\{\frac{v_{n+1}}{v_{n}}\right\}_{n\in\N_0} \right\|_\infty \|x\|_{p,v}.
\end{align*}
It follows that $\|B(r,s)\| \leq \left(|r|+|s|\| \{\frac{v_{n+1}}{v_{n}}\}_{n\in\N_0}\|_\infty\right)$. 

Conversely, if  $B(r,s)\in \cL(l_p(v))$ there exists $C>0$ such that $\|B(r,s)x\|_{p,v}\leq C \|x\|_{p,v}.$  For each $k\in\N_0$ we consider the sequence $e_k=(\delta_{n,k})_{n\in\N_0}\in l_p(v)$, that satisfies in particular $\|e_k\|_{p,v}=v_k$. Then an easy computation shows that $$|s|v_{k+1}\leq\left(|r|^pv_k^p+|s|^pv_{k+1}^p\right)^{\frac{1}{p}}=\|B(r,s)e_k\|_{p,v}\leq Cv_k$$ 
for every $k\in\N_0$. This implies that $\{\frac{v_{n+1}}{v_{n}}\}_{n\in\N_0}\in l_\infty$.
\end{proof}

\begin{rem}\label{norm}
Let suppose that $\{\frac{v_{n+1}}{v_{n}}\}_{n\in\N_0}\in l_\infty$ and denote by $N:=\|\{\frac{v_{n+1}}{v_{n}}\}_{n\in\N_0}\|_\infty$.
From the proof of Proposition \ref{contbrs}, we deduce an estimate of the operator norm of $B(r,s)$ acting on  $l_p(v)$ in the spirit of Proposition \ref{ineqb}, i.e.:
$$
\left(|r|^p+|s|^p\frac{v_{k+1}^p}{v_k^p}\right)^{\frac{1}{p}}=\frac{\|B(r,s)e_k\|_{p,v}}{\|e_k\|_{p,v}} \leq \|B(r, s )\|\leq \left(|r|+N|s|\right)
$$
 for every $k\in\N_0$. It  follows that
 \[
 \left(|r|^p+N|s|^p\right)^{\frac{1}{p}}\leq \|B(r, s )\|\leq \left(|r|+N|s|\right).
 \] 
\end{rem}

 We now pass to describe the spectrum of the operator $B(r,s)$ when it acts on $l_p(v)$. In order to do this, we need the following result.
 
 \begin{lem}\label{L.lemma} Let $1\leq q\leq \infty$ and $v$ be a weight sequence. The multiplication operator $T_v\colon l_q(v)\to l_q$ defined by $T_v((x_n)_{n\in\N_0}):=(x_nv_n)_{n\in\N_0}$, for $x=(x_n)_{n\in\N_0}\in l_q(v)$, is an isometry onto with inverse $T_v^{-1}\colon l_q\to l_q(v)$, given by $T_v^{-1}((y_n)_{n\in\N_0}):=(\frac{y_n}{v_n})_{n\in\N_0}$, for $y=(y_n)_{n\in\N_0}\in l_q$.
 \end{lem}

\begin{proof}The proof is straightforward. \end{proof}

 \begin{thm}\label{charac spect}
 Let $1<p<\infty$ and  $v$ be a weight sequence such that  $\{\frac{v_{n+1}}{v_{n}}\}_{n\in\N_0}\in l_\infty$. For each  $r,s\in\R$ and $r,s\neq0$, we have that 
 \begin{equation}\label{inc1}
 \sigma(B(r , s), l_p(v))=\left\{\alpha\in\C\colon |r-\alpha|\leq L|s|\right\}=\overline{\D}_{L |s|}(r),
 \end{equation}
where $L:=\max\lim_{n\to\infty}\frac{v_{n+1}}{v_{n}}$. 
 \end{thm}

\begin{proof}
We first prove that $(B(r,s)-\alpha I)^{-1}$ exists and belongs to $\cL(l_1(v))$ for $\alpha\notin \overline{\D}_{L |s|}(r)$. If this is the case, since $s\neq 0$ we have $\alpha \neq r$ and so $B(r,s)-\alpha I $ is triangle; hence, $(B(r,s)-\alpha I)^{-1}$ exists. In view of Lemma \ref{L.lemma}, we have that the operator $(B(r,s)-\alpha I)^{-1}\in \cL(l_1(v))$ if, and only if, the operator $T_v\circ (B(r,s)-\alpha I)^{-1}\circ T_v^{-1}\in \cL(l_1)$. In particular, the operator $T_v\circ (B(r,s)-\alpha I)^{-1}\circ T_v^{-1}$
is represented by the following matrix
$$
\left[\begin{matrix}\frac{1}{r-\alpha}&0&\dots &\dots \\ \frac{-s}{(r-\alpha)^2}\frac{v_1}{v_0}&\frac{1}{r-\alpha}&0&\dots \\ \frac{s^2}{(r-\alpha)^3}\frac{v_2}{v_0}&\frac{-s}{(r-\alpha)^2}\frac{v_2}{v_1}&\frac{1}{r-\alpha}&\dots \\ \vdots&\vdots &\vdots&\ddots \end{matrix}\right],
$$
i.e.,  the $n$-th row turns out to be
$$\frac{(-s)^{n-k}}{(r-\alpha)^{n-k+1}}\frac{v_n}{v_k}$$
in the $k$-th place for $k \leq n$ and zero otherwise. Thus, $T_v\circ (B(r,s)-\alpha I)^{-1}\circ T_v^{-1}\in \cL(l_1)$ if, and only if, $\sup_{k\in\N_0}\sum_{n=k}^\infty\left|\frac{s}{r-\alpha}\right|^{n-k}\frac{1}{|r-\alpha|}\frac{v_n}{v_k}<\infty$ (cf. \cite[p. 220]{Taylor}). So, 
 we first observe that
\begin{equation}\label{eq.D1}
 \sup_{k\in\N_0} \sum_{n=k}^\infty \left| \frac{s}{r-\alpha}\right| ^{n-k}  \frac{1}{|r-\alpha|}\frac{v_n}{v_k} 
=  \frac{1}{|r-\alpha|} \sup_{k\in\N_0} \sum_{l=0}^\infty \left| \frac{s}{r-\alpha}\right| ^{l} \frac{v_{l+k}}{v_k}.
\end{equation}
Since $|r-\alpha|>L|s|$ (i.e., $\frac{|r-\alpha|}{|s|}>L$), with $L=\max\lim_{n\to\infty}\frac{v_{n+1}}{v_n}$, there exists $n_0\in\N_0$ such that   $\frac{|r-\alpha|}{|s|}>\sup_{n\geq n_0}\frac{v_{n+1}}{v_n}$. Accordingly, there exists $a\in\R$ such that 
$\frac{|r-\alpha|}{|s|}>a>\sup_{n\geq n_0}\frac{v_{n+1}}{v_n}$, from which follows that $v_{n+1}<a v_n$ for every $n\geq n_0$. This fact implies for each $k\in\N_0$ with $k\geq n_0$  that
\[
\sum_{l=0}^\infty \left| \frac{s}{r-\alpha}\right| ^{l} \frac{v_{l+k}}{v_k}\leq \sum_{l=0}^\infty \left| \frac{s}{r-\alpha}\right| ^{l}a^l<\infty,
\]
because  $\left| \frac{sa}{r-\alpha}\right| <1$  (note that from  $v_{n+1}< av_n$, for each $n\geq n_0$, it easily follows that  $v_{l+k}\leq a^lv_{k}$ for each $k,l\in\N_0$ with $k\geq n_0$). On the other hand, if $k\in \{0,\ldots, n_0-1\}$ then
\begin{align*}
&\quad \sum_{l=0}^\infty \left| \frac{s}{r-\alpha}\right| ^{l} \frac{v_{l+k}}{v_k}=\sum_{l=0}^{n_0-1}\left| \frac{s}{r-\alpha}\right| ^{l} \frac{v_{l+k}}{v_k}+\sum_{l=n_0}^\infty\left| \frac{s}{r-\alpha}\right| ^{l} \frac{v_{l+k}}{v_{n_{0}}}\frac{v_{n_{0}}}{v_k}\\
&\leq \sup_{k=0,\ldots, n_0-1}\sum_{l=0}^{n_0-1} \left| \frac{s}{r-\alpha}\right| ^{l}\frac{v_{l+k}}{v_k}+\left(\sup_{k=0,\ldots,n_0-1}\frac{v_{n_0}a^{k-n_0}}{v_k}\right)\sum_{l=n_0}^\infty\left| \frac{s}{r-\alpha}\right| ^{l}a^l=:C<\infty.
\end{align*}
Therefore, we can conclude that the supremum in 
 \eqref{eq.D1} is finite.
This shows that $(B(r,s)-\alpha I)^{-1}\in \cL(l_1(v))$. 

In view of Lemma \ref{L.lemma}, we can argue in a similar way to prove  that $(B(r,s)-\alpha I)^{-1}\in \cL(l_\infty(v))$ if, and only if, $\sup_{n\in\N_0}\sum_{k=0}^n\left|\frac{s}{r-\alpha}\right|^{n-k}\frac{1}{|r-\alpha|}\frac{v_n}{v_k}<\infty$ (cf. \cite[p. 220]{Taylor}). So, we  observe that 
\begin{equation}\label{eq.D2}
\sup_{n\in\N_0}\sum_{k=0}^n\left|\frac{s}{r-\alpha}\right|^{n-k}\frac{1}{|r-\alpha|}\frac{v_n}{v_k}= \frac{1}{|r-\alpha|}\sup_{n\in\N_0}\sum_{l=0}^n\left|\frac{s}{r-\alpha}\right|^{l}\frac{v_{n}}{v_{n-l}}.
\end{equation}
Since $\frac{|r-\alpha|}{|s|}>a>\sup_{n\geq n_0}\frac{v_{n+1}}{v_n}$, we have for each $n\geq n_0$ that
\begin{align*}
&\sum_{l=0}^n\left|\frac{s}{r-\alpha}\right|^{l}\frac{v_{n}}{v_{n-l}}=\sum_{l=0}^{n-n_0}\left|\frac{s}{r-\alpha}\right|^{l}\frac{v_{n}}{v_{n-l}}+\sum_{l=n-n_0+1}^n\left|\frac{s}{r-\alpha}\right|^{l}\frac{v_{n}}{v_{n_0}}\frac{v_{n_0}}{v_{n-l}}\\
&\quad \leq \sum_{l=0}^{n-n_0}\left|\frac{s}{r-\alpha}\right|^{l}a^l+\sum_{l=n-n_0+1}^n\left|\frac{s}{r-\alpha}\right|^{l}a^{l}a^{n-l-n_0}\frac{v_{n_0}}{v_{n-l}}\\
&\quad \leq \left(1+\sup_{0\leq j\leq n_0-1}\frac{v_{n_0}}{v_j}a^{j-n_0} \right)\sum_{l=0}^{\infty}\left|\frac{s}{r-\alpha}\right|^{l}a^l<\infty,
\end{align*}
because $\frac{|sa|}{|r-\alpha|}<1$ (note that if $0\leq l\leq n-n_0$ then $n-l\geq n_0$ and that if $n-n_0+1\leq l\leq n$, then $0\leq n-l\leq n_0-1$ and hence $-n_0\leq n-l-n_0\leq -1$). On the other hand, it is clear that 
\[
\sup_{0\leq n\leq n_0-1}\sum_{l=0}^n\left|\frac{s}{r-\alpha}\right|^{l}\frac{v_{n}}{v_{n-l}}<\infty.
\]
Therefore, we can conclude that 	the supremum in \eqref{eq.D2} is also finite and hence $(B(r,s)-\alpha I)^{-1}\in \cL(l_\infty(v))$.

Now, by \cite[Theorem 4.52-B, p. 224]{Taylor} combined with Lemma \ref{L.lemma}, we can conclude that $(B(r,s)-\alpha I)^{-1}\in \cL(l_p(v))$ as $1<p<\infty$. Thus, $\alpha\notin \sigma(B(r , s), l_p(v))$. This shows that an inclusion in \eqref{inc1} is satisfied.

To show the other inclusion in \eqref{inc1}, we first observe that if $\alpha =r$, then the operator $B(r,s)-\alpha I$ coincides with the operator $B(0,s)$. Since $B(0,s)x=0$ clearly implies $x=0$, we have that $B(0,s) \colon l_p(v)\to l_p(v)$ is injective. But, $B(0,s) \colon l_p(v)\to l_p(v)$  is  not surjective because $B(s,0)(l_p(v))\subseteq \{y=(y_n)_{n\in\N_0}\in l_p(v):\ y_0=0\}$. Hence, $B(0,s)$  is not invertible. So, $r\in \sigma(B(r,s),l_p(v))$.

Next, let $\alpha\not\in  \sigma(B(r,s),l_p(v))$ (hence,  $\alpha\not=r$). Then there exists $(B(r,s)-\alpha I)^{-1}\in \cL(l_p(v))$. So, it follows  that $(B(r,s)-\alpha I)^{-1}(e_0)\in l_p(v)$ and hence, that
\[
\frac{1}{|r-\alpha|^p}\sum_{n=0}^\infty\left|\frac{s}{r-\alpha}\right|^{np}v_n^p<\infty.
\]
The convergence of the series above implies that $\max\lim_{n\to\infty}\left|\frac{s}{r-\alpha}\right|^p\left(\frac{v_{n+1}}{v_n}\right)^p\leq 1$ or, equivalently, that $\max\lim_{n\to\infty}\left|\frac{s}{r-\alpha}\right|\frac{v_{n+1}}{v_n}\leq 1$, i.e., $\left|\frac{s}{r-\alpha}\right|\max\lim_{n\to\infty}\frac{v_{n+1}}{v_n}=\left|\frac{s}{r-\alpha}\right|L\leq 1$. Accordingly, $\alpha\not\in \D_{L |s|}(r)$. Since  $\alpha\not\in  \sigma(B(r,s),l_p(v))$ is arbitrary, we conclude that $\rho(B(r,s),l_p(v))\subseteq \D^c_{L |s|}(r)$ and hence that $ \D_{L |s|}(r)\subseteq  \sigma(B(r,s),l_p(v))$. Since $ \sigma(B(r,s),l_p(v))$ is a closed set in $\C$, it follows that $\overline{\D}_{L |s|}(r)\subseteq  \sigma(B(r,s),l_p(v))$. 
This completes the proof.
\end{proof}

\begin{thm}\label{charac spect p}
 Let $1<p<\infty$ and $v$ be a weight sequence such that  $\{\frac{v_{n+1}}{v_{n}}\}_{n\in\N_0}\in l_\infty$. For each $r,s\in\R$ and $r,s\neq 0$, we have that $\sigma_p(B(r , s), l_p(v))=\emptyset$. Moreover, 
 \begin{equation}\label{inc3}
\left\{\alpha\in\C\colon |r-\alpha|< L_1 |s|\right\}\subseteq \sigma_p(B^*(r , s), l^*_p(v)) \subseteq \left\{\alpha\in\C\colon |r-\alpha|\leq L_1 |s|\right\},
 \end{equation}
where $L_1:=\min\lim_{n\to\infty}\frac{v_{n+1}}{v_n}$.
 \end{thm}

 \begin{proof} The fact that the point spectrum $\sigma_p(B(r , s), l_p(v))$ is empty follows by arguing as in \cite[Theorem 2.4]{BF2}.
 	
To prove the validity of the inclusions in \eqref{inc3},  we first observe that $l^*_p(v)=l_{p'}(\frac{1}{v})$, where ${p'}$ is the conjugate exponent of $p$, and that $B^*(r,s)((x_n)_{n\in\N_0})=\left(rx_n+sx_{n+1}\right)_{n\in\N_0}$, for $(x_n)_{n\in\N_0}\in l^*_p(v)$.

Now, we recall from \cite[Theorem 2.5]{BF2} that the unique solution of the equation $B^*(r,s)x=\alpha x$, with $\alpha\in\C$ and $ x\not=0 $, is given by  
\begin{equation}\label{eq.aut}
x_n:=\left(\frac{\alpha-r}{s}\right)^n x_0
\end{equation}
 for each $n\in\N_0$ and some $x_0\in\C$, $x_0\neq0$. 
 
 Now,  if $\alpha \in \D_{L_1 |s|}(r)$,  then $|r-\alpha|<L_1|s|$, i.e.,  $\frac{|r-\alpha|}{|s|}<L_1$, with $L_1=\min\lim_{n\to\infty}\frac{v_{n+1}}{v_n}$. Therefore, there exists $n_0\in\N_0$ such that  $\frac{|r-\alpha|}{|s|}<\inf_{n\geq n_0}\frac{v_{n+1}}{v_n}$. It follows that there exists $b>0$ such that $\frac{|r-\alpha|}{|s|}<b<\frac{v_{n+1}}{v_n}$ for every $n\geq n_0$. This fact implies that  
 the element $x=(x_n)_{n\in\N_0}$, with each $x_n$ defined as in \eqref{eq.aut}, belongs to  $l_{p'}(\frac{1}{v})$. Indeed, we have
\begin{align*}\label{eq.D3}
&\sum_{n=0}^\infty \left| \frac{\alpha-r}{s}\right| ^{np'}\frac{|x_0|^{p'}}{ v_n^{p'}}=\sum_{n=0}^{n_0-1}\left| \frac{\alpha-r}{s}\right| ^{np'}\frac{|x_0|^{p'}}{ v_n^{p'}}+\sum_{n=n_0}^\infty \left| \frac{\alpha-r}{s}\right| ^{np'}\frac{|x_0|^{p'}}{ v_n^{p'}}\\
& \leq\sum_{n=0}^{n_0-1}\left| \frac{\alpha-r}{s}\right| ^{np'}\frac{|x_0|^{p'}}{ v_n^{p'}}+\frac{b^{n_0}}{v_{n_0}}\sum_{n=n_0}^\infty \left| \frac{\alpha-r}{sb}\right| ^{np'}<\infty,
\end{align*}
because $\alpha$ satisfies the condition $\left| \frac{\alpha-r}{sb}\right| < 1$, that implies $\left| \frac{\alpha-r}{sb}\right| ^{p'}<1$ (note that $b<\frac{v_{n+1}}{v_n}$ for $n\geq n_0$ implies that $\frac{1}{v_{n+1}}<\frac{1}{b}\frac{1}{v_n}$ for $n\geq n_0$). By the arbitrariness of $\alpha$, we conclude that $\D_{L_1|s|}(r)\subseteq \sigma_p(B^*(r , s), l^*_p(v))$. 

On the other hand, if the previous series converges, then $\max\lim_{n\to\infty}\left| \frac{\alpha-r}{s}\right| ^{p'}\left(\frac{v_n}{ v_{n+1}}\right)^{p'}\leq 1$ or, equivalently, $\max\lim_{n\to\infty}\left| \frac{\alpha-r}{s}\right| \frac{v_n}{ v_{n+1}}\leq 1$, i.e., $\left|\frac{\alpha-r}{s}\right|\max\lim_{n\to\infty} \frac{v_n}{ v_{n+1}}\leq 1$. But, $\max\lim_{n\to\infty} \frac{v_n}{ v_{n+1}}=\frac{1}{\min\lim_{n\to\infty}\frac{v_{n+1}}{v_n}}=\frac{1}{L_1}$. So, we get that $|r-\alpha|\leq L_1|s|$, i.e., $\alpha\in \overline{\D}_{L_1|s|}(r)$. The proof is now complete.
\end{proof}

\begin{rem}\label{oss1}
If $\alpha\in\sigma_p(B^*(r , s), l^*_p(v))$, then dim Ker$(B^*(r , s)-\alpha I)=1$. In particular, the proof of Theorem \ref{charac spect p} shows that $\alpha\in\sigma_p(B^*(r , s), l^*_p(v))$ if, and only if, $\left\{\left(\frac{\alpha-r}{s}\right)^n\right\}_{n\in\N_0}\in  l_{p'}(\frac{1}{v})$.
\end{rem}

We recall that if $T\in\cL(X)$, with $X$ a complex Banach space, then $\sigma_p(T,X)\subseteq \sigma_p(T^*,X^*)\cup \sigma_r(T^*,X^*)$. In particular, if $\sigma_p(T^*,X^*)=\emptyset$, then $\sigma_p(T,X)\subseteq  \sigma_r(T^*,X^*)$. On the other hand, we always have that $\sigma_r(T,X)\subseteq  \sigma_p(T^*,X^*)$. In view of these facts and Theorem \ref{charac spect p}, we easily  obtain the following description of $\sigma_r(B(r , s), l_p(v))$

\begin{thm}\label{charac spect r}
 Let $1<p<\infty$ and $v$ be a weight sequence such that $\{\frac{v_{n+1}}{v_{n}}\}_{n\in\N_0}\in l_\infty$. For each $r,s\in\R$ and $r,s\neq0$  the following inclusions are satisfied:
 \begin{equation*}
\left\{\alpha\in\C\colon |r-\alpha|< L_1 |s|\right\}\subseteq \sigma_r(B(r , s), l_p(v)) \subseteq \left\{\alpha\in\C\colon |r-\alpha|\leq L_1 |s|\right\},
 \end{equation*}
where $L_1:=\min\lim_{n\to\infty}\frac{v_{n+1}}{v_n}$. 
 In particular, $\sigma_r(B(r , s), l_p(v))=\sigma_p(B^*(r , s), l^*_p(v))$.
 \end{thm}

 \begin{proof}
 Since $l_p(v)$ is a reflexive Banach space, we have   $B^{**}(r,s)=B(r,s)$. Therefore, in view of the fact that $\sigma_p(B(r,s),l_p(v))=\emptyset$, it follows that
 \begin{align*}
 \sigma_p(B^*(r , s), l^*_p(v))\subseteq
 \sigma_r(B(r , s), l_p(v)) .
 \end{align*}
So, applying Theorem \ref{charac spect p}, we get that
  \begin{align*}
 \left\{\alpha\in\C\colon |r-\alpha|< L_1 |s|\right\}&\subseteq\sigma_r(B(r , s), l_p(v))\subseteq \sigma_p(B^*(r , s), l^*_p(v)) \\&\subseteq \left\{\alpha\in\C\colon |r-\alpha|\leq L_1 |s|\right\}.
\end{align*}
In particular, we have that $\sigma_r(B(r , s), l_p(v))=\sigma_p(B^*(r , s), l^*_p(v))$.
\end{proof}

We now discuss the continuous spectrum. So, we recall that if $T\in\cL(X)$, with $X$ a complex Banach space, then $\sigma(T,X)= \sigma_p(T,X)\cup \sigma_r(T,X)\cup\sigma_c(T,X)$, with a disjoint union.

\begin{thm}\label{charac spect c}
 Let $1<p<\infty$ and $v$ be a weight sequence such that  $\{\frac{v_{n+1}}{v_{n}}\}_{n\in\N_0}\in l_\infty$. For each $r,s\in\R$ and $r,s\neq0$ the following inclusion is satisfied: 
 \begin{equation}\label{inc5}
 	\sigma(B(r,s),l_p(v))\setminus \overline{\D}_{L_1|s|}(r)
\subseteq \sigma_c(B(r , s), l_p(v)).
 \end{equation}
 \end{thm}
 \begin{proof} Since $\sigma_p(B(r,s),l_p(v))=\emptyset$, we have that
  $$
  \sigma(B(r , s), l_p(v))=\sigma_c(B(r , s), l_p(v))\cup \sigma_r(B(r , s), l_p(v)),
  $$
   with a disjoint union. On the other hand, by Theorem  \ref{charac spect r}, the following inclusions 
    $$
    \left\{\alpha\in\C\colon |r-\alpha|<  L_1|s|\right\}\subseteq\sigma_r(B(r , s), l_p(v))\subseteq\left\{\alpha\in\C\colon |r-\alpha|\leq  L_1|s|\right\}
    $$ 
    are satisfied.
    Thus, it necessarily follows that 
    $$
   \sigma(B(r,s), l_p(v))\setminus \overline{\D}_{L_1|s|}(r)\subseteq\sigma_c(B(r , s), l_p(v)).
    $$
 This means that the inclusion \eqref{inc5} holds.
\end{proof}

\begin{rem} 
(1)	If $v=\mathbbm{1}$, then $L=L_1=1$. Therefore,  our results recover those contained in \cite{BF2}.

(2) If $v$ is a weight sequence satisfying  $\min\lim_{n\to\infty}\frac{v_{n+1}}{v_n}=0$ and $\max\lim_{n\to\infty}\frac{v_{n+1}}{v_n}=L>0$, then for each $r,s\in\R$ with $r,s\not=0$ and $1<p<\infty$ we have that $\sigma_r(B(r,s),l_p(v))=\{r\}$,   $\sigma_c(B(r,s),l_p(v))=\overline{\D}_{L|s|}(r)\setminus\{r\}$ and $\sigma(B(r,s),l_p(v))=\overline{\D}_{L|s|}(r)$ (cf. Theorems \ref{charac spect}, \ref{charac spect r} and \ref{charac spect c}).

(3) If $v$ is a weight sequence satisfying  $\lim_{n\to\infty}\frac{v_{n+1}}{v_n}=0$, then for each $r,s\in\R$ with $r,s\not=0$ and $1<p<\infty$ we have that $\sigma_r(B(r,s),l_p(v))=\sigma(B(r,s),l_p(v))=\{r\}$ and $\sigma_c(B(r,s),l_p(v))=\emptyset$ (cf. Theorems \ref{charac spect}, \ref{charac spect r} and \ref{charac spect c}). 
\end{rem}

\begin{ex}\label{E.2} Let $\alpha=\{\alpha_n\}_{n\in\N_0}\subseteq \R$ be a non-negative increasing sequence with $\lim_{n\to\infty}\alpha_n=\infty$ and $\lim_{n\to\infty}(\alpha_{n+1}-\alpha_n)=:l\in [0,\infty)$. For some $a >0$ and for each $n\in\N_0$,
set $v_n:= a^{\alpha_n}$. Then $\frac{v_{n+1}}{v_n}=a^{\alpha_{n+1}-\alpha_n},$ for every $n\in\N_0$. Accordingly, $N=a^{\sup_{n\in\N_0}(\alpha_{n+1}-\alpha_n)}<\infty$ and hence $B(r,s)\in \cL(l_p(v))$ for every $r,s\in \R$ with $r,s\not=0$. Moreover, $L=L_1=a^l$. So, by  Theorem \ref{charac spect}, we have that 
$$
\sigma(B(r , s), l_p(v))= \left\{\alpha\in\C\colon |r-\alpha|\leq a^l |s|\right\}.
$$
Furthermore, by Theorem \ref{charac spect r} we have that 
$$
\left\{\alpha\in\C\colon |r-\alpha|  < a^l|s|\right\}\subseteq \sigma_r(B(r , s), l_p(v)) \subseteq\left\{\alpha\in\C\colon |r-\alpha|\leq  a^l|s|\right\}.
$$
More precisely, we have that 
 $$
 \sigma_r(B(r , s), l_p(v))=\left\{\alpha\in\C\colon |r-\alpha|  \leq a^l|s|\right\}
 $$
 if, and only if, the series $\sum_{n=0}^\infty a^{(ln-\alpha_n)p'}<\infty$.  Indeed, by Remark \ref{oss1} if $\alpha\in\C$ satisfies $|r-\alpha|=a^l|s|$, then $\alpha\in \sigma_r((B(r , s), l_p(v))$ if, and only if, the element $\{(\frac{\alpha-r}{s})^n\}_{n\in\N_0}\in l_{p'}(\frac{1}{v})$, where $\frac{1}{v}=\{\frac{1}{a^{\alpha_n}}\}_{n\in\N_0}$.
\end{ex}

We conclude this section with results on the ergodic properties of the operator 
$B(r,s)$
acting on the weighted Banach space 
 $l_p(v)$. 

\begin{prop}\label{melpv}
Let $1<p<\infty$ and $v$ be a weight sequence such that $\{\frac{v_{n+1}}{v_{n}}\}_{n\in\N_0}\in l_\infty$. For given $r,s\in\R$ and $r,s\neq0$, consider the following statements:
\begin{enumerate}
	\item $|r|+N|s|\leq 1$;
	\item $B(r,s)$ is power bounded in $\cL(l_p(v))$;
\item $B(r,s)$ is mean ergodic in $\cL(l_p(v))$; 
\item $\frac{B^n(r,s)}{n}\to0$ in $\cL_s(l_p(v))$ as $n\to\infty$;
\item $|r|+L|s|\leq 1$,
\end{enumerate}
where $N:=\|\{\frac{v_{n+1}}{v_{n}}\}_{n\in\N_0}\|_\infty$ and $L:=\max\lim_{n\to\infty}\frac{v_{n+1}}{v_{n}}$.
Then $(1)\Rightarrow(2)\Rightarrow(3)\Rightarrow(4)\Rightarrow(5)$.

Finally, if $L_1:=\min\lim_{n\to\infty}\frac{v_{n+1}}{v_{n}}>0$, then the operator $B(r,s)\in\cL(l_p(v))$ is not supercyclic.
\end{prop}
\begin{proof} The proof of (1)$\Rightarrow $(2) follows from Remark \ref{norm}. The proof of (2)$\Rightarrow$(3)$\Rightarrow $(4)$\Rightarrow $(5) follows by arguing as in the proof  of Proposition \ref{me}, in view of Lemma \ref{incspe}.

	Since $\D_{L_1|s|}(r)\subseteq \sigma_p(B^*(r,s),l^*_p)$ and  $\D_{L_1|s|}(r)$ is an open disc of center $r$ and radius $L_1|s|>0$ (as $L_1>0$), the operator $B(r,s)$ cannot be supercyclic by \cite[Proposition 1.26]{BM}.
\end{proof}

\begin{rem} Note that the conditions (1)$\div$(5) in Proposition \ref{melpv} are all equivalent whenever $N=L$ as in Proposition \ref{me}.
\end{rem}

\begin{prop}\label{meulpv} 
Let $1<p<\infty$ and $v$ be a weight sequence such that $\{\frac{v_{n+1}}{v_{n}}\}_{n\in\N_0}\in l_\infty$.
For each  $r,s\in\R$ and $r,s\neq0$, consider the following statements:
\begin{enumerate}
\item $|r|+L|s|< 1$;
\item $B^n(r,s)\to0$ in $\cL_b(l_p(v))$  as $n\to\infty$;
\item $B(r,s)$ is uniformly mean ergodic in $\cL(l_p(v))$;
\item $1\not\in \sigma(B(r,s),l_p(v))$,
\end{enumerate}
where $L:=\max\lim_{n\to\infty}\frac{v_{n+1}}{v_n}$. Then $(1)\Rightarrow(2)\Rightarrow(3)\Rightarrow(4)$. 
\end{prop}

\begin{proof} 
(1)$\Rightarrow$(2): Recall that $\lim_{n\to\infty}\|B^n(r,s)\|^{1/n}=\sup\{|\alpha|\colon \alpha\in\sigma(B(r,s),l_p(v))\}$. But $\sup\{|\alpha|\colon \alpha\in\sigma(B(r,s),l_p(v))\}=|r|+L|s|$ in view of Theorem \ref{charac spect}. So, there exist $a\in\R$ satisfying $|r|+L|s|<a<1$ and $n_0\in\N$ such that $\|B^n(r,s)\|^{1/n}<a$ for every $n\geq n_0$. It follows that $\|B^n(r,s)\|<a^n$ for every $n\geq n_0$. This implies that $B^n(r,s)\to0$ in $\cL_b(l_p(v))$  as $n\to\infty$.

(2)$\Rightarrow$(3): This fact follows by  arguing as in the proof  of Proposition  \ref{ume}. 

(3)$\Rightarrow$(4): The proof follows in a similar way as in the proof of (3)$\Rightarrow$(4) of Proposition \ref{ume}. For the sake of completeness, we present it.

The uniform mean ergodicity implies that $(I-B(r,p))(l_p(v))$ is a closed subspace of $l_p(v)$  and 
\begin{equation}\label{eq.sommav}
	(I-B(r,s))(l_p(v))\oplus \Ker(I-B(r,s))=l_p(v),
\end{equation}
\cite{Lin} (see also \cite[Theorem 5.4]{ABR-0}). Since $\sigma_p(B(r , s), l_p(v))=\emptyset$, we have that $\Ker(I-B(r,s))=\{0\}$. In view of  \eqref{eq.sommav} it follows that $(I-B(r,s))(l_p(v))=l_p(v)$ and hence, $(I-B(r,s))\colon l_p(v)\to l_p(v)$ is bijective. Accordingly,  $1\notin \sigma(B(r,s), l_p(v))=\overline{\D}_{L|s|}(r)$.
\end{proof}

\begin{ex} Let $v_n=a^n$ for some $a>0$ and every $n\in\N_0$ (cf. Example \ref{E.2}).  Since $N=L=a$, by  Proposition \ref{melpv} we have  for each $r,s\in\R$ and $r,s\neq0$ that the operator $B(r,s)$ is mean ergodic if, and only if, it is power bounded if, and only if, $|r|+a|s|\leq 1$. On the other hand, by Proposition \ref{meulpv} the operator $B(r,s)$ is uniformly  mean ergodic whenever $|r|+a|s|< 1$.
\end{ex}

\section{$B(r,s)$ acting on power series spaces $\Lambda_\infty(\alpha)$ of infinite type  and on their strong duals $\Lambda_\infty(\alpha)'$}

Let $\alpha=\{\alpha_n\}_{n\in\N_0}\subseteq \R$ be a non-negative increasing sequence with $\lim_{n\to\infty}\alpha_n=\infty$. 
The \textit{power series space $\Lambda_\infty(\alpha)$  of  infinite type associated to $\alpha$} is defined by 
\begin{equation*}\label{eq.spazios}
	\Lambda_\infty(\alpha):=\left\{x\in \omega\,\colon \, |x|_k^2:=\sum_{n=0}^\infty |x_n|^2e^{2k\alpha_n}<+\infty,\ \forall k\in\N\right\}.
\end{equation*}  
The sequence space $\Lambda_\infty(\alpha)$,  equipped with  the locally convex topology generated by the sequence $(|\cdot|_k)_{k\in\N}$ of norms, is a  Fr\'echet Schwartz  space (hence, a Fr\'echet Montel space); see \cite[p. 357]{MV}. The  power series space $\Lambda_\infty(\alpha)$ of infinite type is  nuclear if, and only if, $\sup_{n\in\N_0}\frac{\log(n+1)}{\alpha_{n+1}}<\infty$ (cf. \cite[Proposition 29.6(1)]{MV}). 

 Power series spaces of infinite type form an important class of K\"othe spaces. For instance, the space of holomorphic functions on $\C$ and the space $\cs$ of rapidly decreasing sequences are examples of power  series spaces of infinite type. Indeed, $\Lambda_\infty((n+1)_{n\in\N_0})\simeq H(\C)$ and $ \Lambda_\infty((\log(n+1))_{n\in\N_0})\simeq \cs$. For the case $H(\C)$, the isomorphism is provided by $T\colon \Lambda_\infty((n+1)_{n\in\N_0})\to H(\C)$, $(x_n)_{n\in\N_0}\mapsto \sum_{n=0}^\infty x_nz^{n}$.

Note that if we set $v_k(n):=e^{k\alpha_n}$ for every $k\in\N$ and $n\in\N_0$, then 
 $\Lambda_\infty(\alpha)$ is the projective limit of the sequence of Banach spaces $l_2(v_k)$, for $k\in\N$, with $l_2(v_{k+1})\hookrightarrow l_2(v_k)$ continuously, i.e., $\Lambda_\infty(\alpha)=\cap_{k\in\N}l_2(v_k)$ and $|\cdot|_k=\|\cdot\|_{2,v_k}$ for every $k\in\N$. Therefore, the strong dual $\Lambda_\infty(\alpha)'$ of $\Lambda_\infty(\alpha)$ is given by $\Lambda_\infty(\alpha)'=\bigcup_{k\in\N} l_2\left(\frac{1}{v_k}\right)$ (i.e., $\Lambda_\infty(\alpha)'$ is the inductive limit, briefly an (LB)-space, of the sequence of Banach spaces $l_2\left(\frac{1}{v_k}\right)$, $k\in\N$, with $l_2\left(\frac{1}{v_k}\right)\hookrightarrow l_2\left(\frac{1}{v_{k+1}}\right)$ continuously). Since $\Lambda_\infty(\alpha)$ is a Fr\'echet Schwartz space (hence, reflexive), its strong dual $\Lambda_\infty(\alpha)'$ is a complete and reflexive (LB)-space. Thus, $(\Lambda_\infty(\alpha)')'_\beta= \Lambda_\infty(\alpha)$.

 The aim of this section is to show that for $r,s\in\R$ and $r,s\neq0$, the operator $B(r,s)$ acts continuously on both the spaces $\Lambda_\infty(\alpha)$ and  $\Lambda_\infty(\alpha)'$, and to describe its spectrum in each setting.

 For the convenience of the reader we recall some results that will be necessary to determine the spectra of the operator $B(r,s)$ acting on $\Lambda_\infty(\alpha)$ or  $\Lambda_\infty(\alpha)'$. To this end, we first observe that  the definitions of spectrum, resolvent, as well as those of point, residual, and continuous spectrum, do not differ from those in the Banach space setting. But some authors (eg. \cite{Wel}) prefer, for $T$ a continuous linear operator over a locally Hausdorff convex space $X$ (briefly, lcHs), to introduce the subset $\rho^*(T,X)$ of $\rho(T,X)$ consisting of all $\lambda\in\C$ for which there exists $\delta>0$ such that the open disc $\D_\delta(\lambda)\subseteq  \rho(T,X)$ and $\{R(\mu, T) \colon \mu\in \D_\delta(\lambda)\}$ is an equicontinuous
 subset of $\cL(X)$. The set $\sigma^*(T,X) := \C \setminus \rho^*(T,X)$, which is a closed set with
 $\sigma(T,X)\subseteq \sigma^*(T,X)$, is called the \textit{Waelbrock spectrum} of $T$. If
 $X$ is a Banach space, then $\sigma(T,X)= \sigma^*(T,X)$.
 
 We also point out that  the definition of (uniformly) mean ergodic operators as well as of (hypercyclic) supercyclic operators in the lcHs  setting  do not differ from those in the Banach space setting.
 
 In the following, we will denote by $T'$ the adjoint (the dual operator) of $T\in\cL(X)$, with $X$ a lcHs, as usual in the setting of non-normable spaces. For all undefined definitions for non-normable spaces, we refer to \cite{MV}.

  For a reflexive  lcHs $X$ and $T \in  \cL(X)$, the following result given in \cite[Lemma 3.1]{Albanese} shows that the spectra of $T$ and of $T'\in \cL(X'_\beta)$  are related, as in the Banach space setting.
 
\begin{lem}\label{L31} Let $X$ be a reflexive and complete lcHs with $X'_\beta$ a complete lcHs and $T \in  \cL(X)$. Then
 \[
 \sigma(T',X'_\beta) = \sigma(T,X) \mbox{ and } \sigma^*(T',X'_\beta)=\sigma^*(T,X).\]
 \end{lem}

The next result given in \cite[Lemma 2.1]{ABR-2} (cf. also \cite[Lemma 2.5]{ABR-1}), will be necessary to determine the spectra of $B(r,s)$ acting on $\Lambda_\infty(\alpha)$.

\begin{lem}\label{L32} Let $X =\cap_{k\in\N}X_k$ be a Fr\'echet space which is the intersection of a sequence of
Banach spaces $(X_k,\|\cdot\|_k)_{k\in\N}$, satisfying $X_{k+1}\subseteq X_k$ with $\|x\|_k\leq \|x\|_{k+1}$ for each
$k\in\N$ and $x\in X_{k+1}$. Let $T\in  \cL(X)$ satisfy the following condition:

(A) For each $k\in\N$ there exists $T_k\in \cL(X_k)$ such that the restriction of $T_k$  to X (resp. of $T_k$  to $X_{k+1}$) coincides with $T$ (resp. with $T_{k+1}$).

Then the following properties are satisfied:
\begin{itemize}
\item[\rm (i)] $\sigma (T, X)\subseteq \cup_{k\in\N}\sigma(T_k,X_k)$ and $\sigma_p (T , X) \subseteq \cap_{k\in\N}\sigma_p(T_k,X_k)$;
\item[\rm (ii)] If  $\cup_{k\in\N}\sigma(T_k,X_k)\subseteq \overline{\sigma(T, X)}$, then $\sigma^*(T,X) = \overline{\sigma(T, X)}$;
\item[\rm (iii)] If ${\rm dim} \Ker(T_k-\lambda I) = 1$ for each $\lambda \in \cap_{k\in\N}\sigma_p(T_k,X_k)$  and $k\in\N$, then $\sigma_p(T, X) =\cap_{k\in\N}\sigma_p(T_k,X_k)$.
\end{itemize}
\end{lem}

We now give the complete characterization of the spectra of the operator $B(r,s)$ acting on  $\Lambda_\infty(\alpha)$.

\begin{thm}\label{T.s} Let $\alpha=\{\alpha_n\}_{n\in\N_0}\subseteq \R$ be a non-negative increasing sequence satisfying  the conditions  $\lim_{n\to\infty}\alpha_n=\infty$,  $\lim_{n\to\infty}(\alpha_{n+1}-\alpha_n)=:l\in [0,\infty)$ and $\alpha_n\geq c \log(n+1)$ for every $n\in\N_0$ and some $c>0$. Then
	for each $r,s\in\R$ and $r,s\neq0$, the operator $B(r,s)\in \cL(\Lambda_\infty(\alpha))$. Moreover, the following properties are satisfied.
	\begin{itemize}
\item[(1)] If $l=0$, then
	\begin{align}
		&\sigma_p(B(r,s),\Lambda_\infty(\alpha))=\sigma_c(B(r,s),\Lambda_\infty(\alpha))=\emptyset,\label{eq.Spcs} \\
		 \sigma^*(B(r,s), \Lambda_\infty(\alpha))&=\sigma(B(r,s),\Lambda_\infty(\alpha))=\sigma_r(B(r,s),\Lambda_\infty(\alpha))=\overline{\D}_{|s|}(r)\label{eq.SrW}
		\end{align}
	\item[ (2)] If $l\in (0,\infty)$, then
	\begin{align}
		&\sigma_p(B(r,s),\Lambda_\infty(\alpha))=\sigma_c(B(r,s),\Lambda_\infty(\alpha))=\emptyset,\label{eq.Spcs} \\
\quad	\sigma^*(B(r,s),  \Lambda_\infty(\alpha))&=\sigma(B(r,s),\Lambda_\infty(\alpha))=\sigma_r(B(r,s),\Lambda_\infty(\alpha))=\C.\label{eq.SrWW}
\end{align}
	\end{itemize}
	\end{thm}

\begin{proof} We first observe that $\frac{v_k(n+1)}{v_k(n)}=e^{k(\alpha_{n+1}-\alpha_n)}$ for each $k\in\N$ and $n\in\N_0$. Since there exists $\lim_{n\to\infty}(\alpha_{n+1}-\alpha_n)=l\in[0,\infty)$, for each $k\in\N$ we have $N_k:=\sup_{n\in\N_0}\frac{v_k(n+1)}{v_k(n)}=e^{k\sup_{n\in\N_0}(\alpha_{n+1}-\alpha_n)}<\infty$. By Proposition \ref{contbrs} and Remark \ref{norm}, it follows for each $k\in\N$ that $B(r,s)\in \cL(l_2(v_k))$ with operator norm less or equal to $|r|+N_k|s|$. Since $\Lambda_\infty(\alpha)=\cap_{k\in\N}l_2(v_k)$ is the projective limit of the sequence $(l_2(v_k))_{k\in\N}$ of Banach spaces, this implies that the operator $B(r,s)\in \cL(\Lambda_\infty(\alpha))$. 
	
	(1) For each $k\in\N$ we have that  $\lim_{n\to\infty}\frac{v_k(n+1)}{v_k(n)}=\lim_{n\to\infty}e^{k(\alpha_{n+1}-\alpha_n)}=e^0=1$. Accordingly, $\max\lim_{n\to\infty}\frac{v_k(n+1)}{v_k(n)}=\min\lim_{n\to\infty}\frac{v_k(n+1)}{v_k(n)}=1$ for every $k\in\N$. On the other hand, the series $\sum_{n=0}^\infty e^{-2k\alpha_n}<\infty$ for  $k\in\N$ large enough. Indeed, for a fixed $k\in\N$,  observe that the condition $\alpha_n\geq c\log(n+1)$ for every $n\in\N_0$ implies that $e^{-k\alpha_n}\leq e^{-kc\log (n+1)}=\frac{1}{(n+1)^{kc}}$, where the series $\sum_{n=0}^\infty\frac{1}{(n+1)^{kc}}<\infty$ whenever $kc>1$, i.e., $k>\frac{1}{c}$. So, 
	in view of Example \ref{E.2} with $a:=e^k$, for $k\in\N$, (cf. also the results in \S\ref{lpv}), we have  for each $k\in\N$ with $k>k_0:=[\frac{1}{c}]+1$ that 
	\begin{equation}\label{eq.K}
	\sigma(B(r,s),l_2(v_k))=	\sigma_r(B(r,s),l_2(v_k))=\overline{\D}_{|s|}(r)
		\end{equation}
		and
		\begin{equation}\label{eq.K1}
	\sigma_p(B(r,s),l_2(v_k))=\sigma_c(B(r,s),l_2(v_k))=\emptyset.
		\end{equation}
	Since $\Lambda_\infty(\alpha)=\cap_{k\geq k_0}l_2(v_k)$, we can apply  Lemma \ref{L32}(i) to obtain that 
	\[
	\sigma_p(B(r,s),\Lambda_\infty(\alpha))=\emptyset,\quad \sigma(B(r,s),\Lambda_\infty(\alpha))\subseteq \overline{\D}_{|s|}(r).
	\]
	To show that $\overline{\D}_{|s|}(r)\subseteq\sigma(B(r,s),\Lambda_\infty(\alpha))$, let $\lambda\in \overline{\D}_{|s|}(r)$ be fixed and suppose that $\lambda\not\in \sigma(B(r,s),\Lambda_\infty(\alpha))$. Then there exists $(B(r,s)-\lambda I)^{-1}\in \cL(\Lambda_\infty(\alpha))$, i.e., the continuous linear operator $(B(r,s)-\lambda I)\colon \Lambda_\infty(\alpha)\to \Lambda_\infty(\alpha) $ is bijective. Thus, for each $n\in\N_0$ there exists $y_n\in\Lambda_\infty(\alpha)$ such that $(B(r,s)-\lambda I)(y_n)=e_n$. It  follows for any $k\in\N$ with $k\geq k_0$ that $\{y_n\colon n\in\N_0\}\subseteq l_2(v_k)$ and so, $(B(r,s)-\lambda I)({\rm span} \{y_n\colon n\in\N_0\})={\rm span}\{e_n\colon n\in\N_0\}$ in $l_2(v_k)$. Accordingly, we obtain that
	\[
{\rm span}\{e_n\colon n\in\N_0\}=(B(r,s)-\lambda I)({\rm span} \{y_n\colon n\in\N_0\})\subseteq (B(r,s)-\lambda I)(l_2(v_k)).
	\]
	Since $\overline{{\rm span}\{e_n\colon n\in\N_0\}}^{l_2(v_k)}=l_2(v_k)$, 
	this implies that 
	\[
	l_2(v_k)=\overline{{\rm span}\{e_n\colon n\in\N_0\}}^{l_2(v_k)}\subseteq  \overline{(B(r,s)-\lambda I)(l_2(v_k))}^{l_2(v_k)}\subseteq l_2(v_k);
	\]
hence $\overline{(B(r,s)-\lambda I)(l_2(v_k))}^{l_2(v_k)}=l_2(v_k)$. This is a contradiction with the fact that  $\lambda\in \sigma_r(B(r,s),l_2(v_k))$ (cf. \eqref{eq.K}). Thus, we can conclude that $\lambda\in \sigma(B(r,s),\Lambda_\infty(\alpha))$ and hence that $\sigma(B(r,s),\Lambda_\infty(\alpha))=\overline{\D}_{|s|}(r)$.  In view of \eqref{eq.K}, it follows via Lemma \ref{L32}(ii)  that $\sigma^*(B(r,s),\Lambda_\infty(\alpha))=\sigma(B(r,s),\Lambda_\infty(\alpha))$.

Finally, let $\lambda\in \sigma(B(r,s),\Lambda_\infty(\alpha))=\overline{\D}_{|s|}(r)$. Suppose that $\lambda\in \sigma_c(B(r,s),\Lambda_\infty(\alpha))$, i.e.,  that
 $
\overline{(B(r,s)-\lambda I)(\Lambda_\infty(\alpha))}^{\Lambda_\infty(\alpha)}=\Lambda_\infty(\alpha)$. 
Since  $X:={\rm span}\{e_n\colon n\in\N_0\}$ is a dense subspace of   $\Lambda_\infty(\alpha)$ and $\overline{(B(r,s)-\lambda  I)(X)}^{\Lambda_\infty(\alpha)}=\overline{(B(r,s)-\lambda  I)(\overline{X}^{\Lambda_\infty(\alpha)})}^{\Lambda_\infty(\alpha)}$, it follows that
\begin{equation}\label{eq.NN}
\overline{(B(r,s)-\lambda  I)({\rm span}\{e_n\colon n\in\N_0\})}^{\Lambda_\infty(\alpha)}=\overline{(B(r,s)-\lambda I)(\Lambda_\infty(\alpha))}^{\Lambda_\infty(\alpha)}=\Lambda_\infty(\alpha).
\end{equation}
On the other hand,   $\overline{{\rm span}\{e_n\colon n\in\N_0\}}^{l_2(v_k)}=l_2(v_k)$ and $\overline{\Lambda_\infty(\alpha)}^{l_2(v_k)}=l_2(v_k)$ for every $k\in\N$. So, by \eqref{eq.NN} we obtain that
$$
\overline{(B(r,s)-\lambda  I)({\rm span}\{e_n:\ n\in\N_0\})}^{l_2(v_k)}=l_2(v_k)
$$
for every $k\in\N$. As $\sigma_{p}(B(r,s), l_2(v_k))=\emptyset$, it follows that $\lambda\in \sigma_c(B(r,s),l_2(v_k))$ for every $k\in\N$; this is a contradiction whenever $k\geq k_0$ (cf. \eqref{eq.K1}). So, we can conclude that $\lambda\in \sigma_r(B(r,s),\Lambda_\infty(\alpha))$, as $\sigma_p(B(r,s),\Lambda_\infty(\alpha))=\emptyset$. This also shows that $\sigma_c(B(r,s),\Lambda_\infty(\alpha))=\emptyset$.

(2) For each $k\in\N$ we have that  $\lim_{n\to\infty}\frac{v_k(n+1)}{v_k(n)}=\lim_{n\to\infty}e^{k(\alpha_{n+1}-\alpha_n)}=e^{kl}$. Therefore,  $\max\lim_{n\to\infty}\frac{v_k(n+1)}{v_k(n)}=\min\lim_{n\to\infty}\frac{v_k(n+1)}{v_k(n)}=e^{kl}$ for every $k\in\N$. So, 
in view of Example \ref{E.2} with $a:=e^k$, for $k\in\N$, (cf. also the results in \S\ref{lpv}), we have  for each $k\in\N$  that 
\begin{equation}\label{eq.KK}
	\sigma(B(r,s),l_2(v_k))=	\overline{\D}_{e^{kl}|s|}(r), \ {\D}_{e^{kl}|s|}(r)\subseteq  \sigma_r(B(r,s),l_2(v_k))\subseteq \overline{\D}_{e^{kl}|s|}(r)
\end{equation}
and
\begin{equation}\label{eq.KK1}
	\sigma_p(B(r,s),l_2(v_k))=\emptyset.
\end{equation}
Since $\Lambda_\infty(\alpha)=\cap_{k\in\N}l_2(v_k)$, we can apply  Lemma \ref{L32}(i) to obtain that 
\[
\sigma_p(B(r,s),\Lambda_\infty(\alpha))=\emptyset,\quad \sigma(B(r,s),\Lambda_\infty(\alpha))\subseteq \cup_{k\in\N}\overline{\D}_{e^{kl}|s|}(r)=\C.
\]
The claim is that  $\C=\sigma(B(r,s),\Lambda_\infty(\alpha))$. To show the claim, let $\lambda\in \C$ be fixed and suppose that $\lambda\not\in \sigma(B(r,s),\Lambda_\infty(\alpha))$. Now, by arguing as in point (i), we obtain  that $\overline{(B(r,s)-\lambda I)(l_2(v_k))}^{l_2(v_k)}=l_2(v_k)$ for every $k\in\N$. This is a contradiction. Indeed,  we also have that
 $\cup_{k\in\N}{\D}_{e^{kl}|s|}(r)=\C$ and hence  there exists $k\in\N$ such that $\lambda\in {\D}_{e^{kl}|s|}(r)= \sigma_r(B(r,s),l_2(v_k))$ (cf. \eqref{eq.KK}, \eqref{eq.KK1}). So, the claim is proved. 
 
 Since $\cup_{k\in\N}\overline{\D}_{e^{kl}|s|}(r)=\C=\sigma(B(r,s),\Lambda_\infty(\alpha))$, in view of  Lemma \ref{L32}(ii) we can conclude that $\sigma^*(B(r,s),\Lambda_\infty(\alpha))=\sigma(B(r,s),\Lambda_\infty(\alpha))$.
 
 To show that $\sigma_r(B(r,s),\Lambda_\infty(\alpha))=\sigma(B(r,s),\Lambda_\infty(\alpha))=\C$, we suppose that $\lambda\in\sigma_c(B(r,s),\Lambda_\infty(\alpha))$ for some $\lambda\in\C$. Now, by arguing again as in point (i), we obtain that $\lambda\in \sigma_c(B(r,s), l_2(v_k))$ for every $k\in\N$. But from $\C=\cup_{k\in\N}\D_{e^{kl}|s|}(r)=\cup_{k\in\N}\sigma_r(B(r,s),l_2(v_k))$ it follows that $\lambda \in \sigma_r(B(r,s),l_2(v_{k_0}))$ for some $k_0\in\N$. But this is a contradiction.
	\end{proof}
	
	\begin{rem}\label{R_Power} Let $\alpha=\{\alpha_n\}_{n\in\N_0}$ be a non-negative increasing sequence satisfying the conditions $\lim_{n\to\infty}\alpha_n=\infty$ and $\lim_{n\to\infty}(\alpha_{n+1}-\alpha_n)=:l\in [0,\infty)$. Then by the Stolz-Ces\`aro 
		criterion (see \cite[Ch. 3, Theorem 1.22]{Mu}) it follows that $\lim_{n\to\infty}\frac{\alpha_n}{n}=l$. Therefore, $\lim_{n\to\infty}\sqrt[n]{e^{kln-k\alpha_n}}=\lim_{n\to\infty}e^{kl-k\frac{\alpha_n}{n}}=1$ for every $k\in\N$. Accordingly,  the root test cannot provide any information regarding the convergence of the series $\sum_{n=0}^\infty e^{kln-k\alpha_n}$, for $k\in\N$ (cf. Example \ref{E.2}). 
		 Hence, the further hypothesis on $\alpha$, i.e., $\alpha_n\geq c \log(n+1)$ for every $n\in\N_0$ and some $c>0$,  has been necessary to ensure the convergence of the series $\sum_{n=0}^\infty e^{kln-k\alpha_n}$ when $l=0$, for $k$ large enough, and hence to describe $\sigma_r(B(r,s), l_2(v_k))$.
		  If $l>0$, the condition $\lim_{n\to\infty}(\alpha_{n+1}-\alpha_n)=l$ implies that $\lim_{n\to\infty}\frac{\log (n+1)}{\alpha_n}=0$ thanks to the Stolz-Ces\`aro
		 criterion. Therefore, in this case, the condition  $\alpha_n\geq c \log(n+1)$ for every $n\in\N_0$ and some $c>0$, is clearly satisfied.
		 
		 Examples of sequences satisfying such an additional condition are  $\{\log (n+1)\}_{n\in\N_0}$ and  $\{n+1\}_{n\in\N_0}$. 
		\end{rem}
	
We now pass to determine the spectra of $B(r,s)$ acting on $\Lambda_\infty(\alpha)'$. In order to do this, we first give this other result, which is proved in \cite[Lemma 5.2]{ABR-3}.

\begin{lem}\label{L33}
 Let $X = \cup_{k\in\N}  X_k$ be an Hausdorff inductive limit of a sequence of Banach spaces
$(X_k,\|\cdot\|_k)_{k\in\N}$. Let $T\in \cL(X)$ satisfy the following condition:

(B) For each $k\in\N$ the restriction $T_k$ of $T$ to $X_k$ maps $X_k$ into itself and $T_k \in\cL(X_k)$.

Then the following properties are satisfied:
\begin{itemize}
\item[\rm (i)] $\sigma_p (T, X) = \cup_{k\in\N}\sigma_p (T_k, X_k)$;
\item[\rm (ii)] If $\cup_{k=m}^\infty \sigma(T_k, X_k)\subseteq  \sigma(T,X)$ for some $m\in\N$, then $\sigma^*(T , X) = \overline{\sigma(T , X)}$;
\item[\rm (iii)] $\sigma (T, X) \subseteq \cap_{m\in\N}(\cup_{k=m}^\infty \sigma(T_k, X_k))$.
\end{itemize}
\end{lem}

\begin{thm}\label{T.s1} Let $\alpha=\{\alpha_n\}_{n\in\N_0}\subseteq \R$ be a non-negative increasing sequence satisfying  the conditions  $\lim_{n\to\infty}\alpha_n=\infty$ and  $\lim_{n\to\infty}(\alpha_{n+1}-\alpha_n)=:l\in [0,\infty)$. 
	 Then
	for each $r,s\in\R$ and $r,s\neq0$, the operator $B(r,s)\in \cL(\Lambda_\infty(\alpha)')$. Moreover, the following properties are satisfied.
	\begin{itemize} 
		\item[(1)] If $l=0$, then
	\begin{align}
	&	\sigma_p(B(r,s),\Lambda_\infty(\alpha)')=\emptyset,\,  \sigma^*(B(r,s), \Lambda_\infty(\alpha)')=\sigma(B(r,s),\Lambda_\infty(\alpha)')=\overline{\D}_{|s|}(r),\label{eq.SpSS} \\ & \quad\quad \sigma_r(B(r,s),\Lambda_\infty(\alpha)')=\D_{|s|}(r),\ \sigma_c(B(r,s),\Lambda_\infty(\alpha)')=\partial \D_{|s|}(r).\label{eq.SprcSS}
		\end{align} 
	\item[(2)] If $l\in (0,\infty)$, then
	\begin{align}
		&\sigma^*(B(r,s), \Lambda_\infty(\alpha))=\sigma(B(r,s), \Lambda_\infty(\alpha))=\sigma_r(B(r,s), \Lambda_\infty(\alpha))=\{r\}\label{eq.SF}\\
		&\quad\quad\sigma_p(B(r,s),\Lambda_\infty(\alpha)')=\sigma_c(B(r,s),\Lambda_\infty(\alpha)')=\emptyset. \label{eq.SFF}
		\end{align}
	\end{itemize}
	\end{thm}
	
	\begin{proof}As was done in Theorem \ref{T.s}, we first observe that $\frac{\frac{1}{v_k(n+1)}}{\frac{1}{v_k(n)}}=\frac{v_k(n)}{v_k(n+1)}=e^{-k(\alpha_{n+1}-\alpha_n)}$ for each $k\in\N$ and $n\in\N_0$ and hence, for each $k\in\N$ we have $N_k':=e^{-k\inf_{n\in\N_0}(\alpha_{n+1}-\alpha_n)}<\infty$, as there exists $\lim_{n\to\infty}(\alpha_{n+1}-\alpha_n)=l\in [0,\infty)$. By Proposition \ref{contbrs} and Remark \ref{norm}, it follows for each $k\in\N$ that $B(r,s)\in \cL\left(l_2\left(\frac{1}{v_k}\right)\right)$ with operator norm less or equal to $|r|+N_k'|s|$. Since $\Lambda_\infty(\alpha)'=\cup_{k\in\N} l_2\left(\frac{1}{v_k}\right)$ is the inductive limit of the sequence $\left(l_2\left(\frac{1}{v_k}\right)\right)_{k\in\N}$ of Banach spaces, this implies that the operator $B(r,s)\in \cL(\Lambda_\infty(\alpha)')$. 
	
(1)	For each $k\in\N$  we have that $\lim_{n\to\infty}\frac{\frac{1}{v_k(n+1)}}{\frac{1}{v_k(n)}}=e^0=1$. Therefore,   $\min\lim_{n\to\infty}\frac{\frac{1}{v_k(n+1)}}{\frac{1}{v_k(n)}}=\max\lim_{n\to\infty}\frac{\frac{1}{v_k(n+1)}}{\frac{1}{v_k(n)}}=1$. On the other hand, the series $\sum_{n=0}^\infty e^{2k\alpha_n}$ clearly diverges for every $k\in\N$ (cf. Example \ref{E.2}). In view of Example \ref{E.2} with $a=e^{-k}$, for $k\in\N$, (cf. also the results in \S\ref{lpv}), it follows  for each $k\in\N$ that 
	\begin{equation}\label{eq.K2}
	\sigma_p \left(B(r,s),l_2\left(\frac{1}{v_k}\right)\right)=\emptyset,\ 	\sigma \left(B(r,s),l_2\left(\frac{1}{v_k}\right)\right)=\overline{\D}_{|s|}(r),
        \end{equation}
        and 
        \begin{equation}\label{eq.K3}
	\sigma_r\left(B(r,s),l_2\left(\frac{1}{v_k}\right)\right)=\D_{|s|}(r),\;\sigma_c\left(B(r,s),l_2\left(\frac{1}{v_k}\right)\right)=\partial\D_{|s|}(r).
        \end{equation}
	So, by Lemma \ref{L33}(i),(iii), we immediately obtain that 
	\[
	\sigma_p(B(r,s),\Lambda_\infty(\alpha)')=\emptyset,\quad \sigma(B(r,s),\Lambda_\infty(\alpha)')\subseteq \overline{\D}_{|s|}(r).
	\]
	To show that $\overline{\D}_{|s|}(r)\subseteq\sigma(B(r,s),\Lambda_\infty(\alpha)')$, we recall that $(\Lambda_\infty(\alpha)')'_\beta=\Lambda_\infty(\alpha)$ and so,  the adjoint (dual) operator $B'(r,s)$ of $B(r,s)\in\cL(\Lambda_\infty(\alpha)')$ belongs to $\cL(\Lambda_\infty(\alpha))$. Now, we note for every $k\in\N$ that the adjoint operator $B^*(r,s)\in \cL(l_2(v_k))$ of $B(r,s)\in \cL(l_2(\frac{1}{v_k}))$ when restricted to $\Lambda_\infty(\alpha)$ coincides with  $B'(r,s)$. Moreover, in view of  Theorem \ref{charac spect r} and \eqref{eq.K3}, we know that $\sigma_p (B^*(r,s),l_2(v_k))=\D_{|s|}(r)$ for every $k\in\N$. So, by Lemma \ref{L32}(iii) combined with Remark \ref{oss1} (i.e., dim Ker$(B^*(r , s)-\lambda I)=1$ if $\lambda\in \sigma_p(B^*(r , s),l_p^*(v_k))$), we get that $\sigma_p (B'(r,s),\Lambda_\infty(\alpha))=\D_{|s|}(r)$. Thanks to   Lemma \ref{L31}  this yields that 
	$$
	\D_{|s|}(r)\subseteq  \sigma (B'(r,s),
		\Lambda_\infty(\alpha))=\sigma (B''(r,s),\Lambda_\infty(\alpha)')=\sigma (B(r,s),\Lambda_\infty(\alpha)').
	$$
	Therefore, it remains to prove that $\partial \D_{|s|}(r)\subseteq  \sigma (B(r,s),\Lambda_\infty(\alpha)').$ So, let $\lambda\in\partial \D_{|s|}(r)$ be fixed. Then by \eqref{eq.K3} we have that $\lambda \in \sigma_c\left(B(r,s),l_2\left(\frac{1}{v_k}\right)\right)$ for every $k\in\N$. This means that $(B(r,s)-\lambda I)\colon l_2\left(\frac{1}{v_k}\right)\to l_2\left(\frac{1}{v_k}\right)$ is injective with dense range for every $k\in\N$. Due to the properties of the inductive topology, this implies that $(B(r,s)-\lambda I)\colon \Lambda_\infty(\alpha)'\to \Lambda_\infty(\alpha)'$ is also injective with dense range. The injectivity is obvious. If $y\in \Lambda_\infty(\alpha)'$, there exists $h\in\N$ such that $y\in l_2\left(\frac{1}{v_h}\right)$. Since $(B(r,s)-\lambda I)\colon l_2\left(\frac{1}{v_h}\right)\to l_2\left(\frac{1}{v_h}\right)$ has dense range,  there exists $(x_n)_{n\in\N_0}\in l_2\left(\frac{1}{v_h}\right)$ such that $(B(r,s)-\lambda I)(x_n)\to y$ in $l_2\left(\frac{1}{v_h}\right)$ as $n\to\infty$. Thus, $(x_n)_{n\in\N_0}$ belongs to $\Lambda_\infty(\alpha)'$ and $(B(r,s)-\lambda I)(x_n)\to y$ in $\Lambda_\infty(\alpha)'$  as $n\to\infty$. This shows that $\lambda \in \sigma_c(B(r,s),\Lambda_\infty(\alpha)')\subseteq \sigma(B(r,s),\Lambda_\infty(\alpha)')$. So, we can conclude that 
	 $$ 
	 \sigma(B(r,s),\Lambda_\infty(\alpha)')= \overline{\D}_{|s|}(r).
	 $$
	In view of \eqref{eq.K2}, by Lemma \ref{L33}(ii) it follows that  $\sigma^*(B(r,s),\Lambda_\infty(\alpha)')= \overline{\D}_{|s|}(r)$. Finally, $\sigma_r(B(r,s),\Lambda_\infty(\alpha)')=\sigma_p (B'(r,s),\Lambda_\infty(\alpha))=\D_{|s|}(r)$ as proved above, and so $\sigma_c(B(r,s),\Lambda_\infty(\alpha)')=\partial \D_{|s|}(r)$.
	
	(2) For each $k\in\N$  we have that $\lim_{n\to\infty}\frac{\frac{1}{v_k(n+1)}}{\frac{1}{v_k(n)}}=e^{-kl}$ and hence  $\min\lim_{n\to\infty}\frac{\frac{1}{v_k(n+1)}}{\frac{1}{v_k(n)}}=\max\lim_{n\to\infty}\frac{\frac{1}{v_k(n+1)}}{\frac{1}{v_k(n)}}=e^{-kl}$. In view of Example \ref{E.2} with $a=e^{-kl}$, for $k\in\N$, (cf. also the results in \S\ref{lpv}), it follows  for each $k\in\N$ that 
	\begin{equation}\label{eq.KK2}
		\sigma_p \left(B(r,s),l_2\left(\frac{1}{v_k}\right)\right)=\emptyset,\ 	\sigma \left(B(r,s),l_2\left(\frac{1}{v_k}\right)\right)=\overline{\D}_{e^{-kl}|s|}(r),
	\end{equation}
	and 
	\begin{equation}\label{eq.KK3}
	\D_{e^{-kl}|s|}(r)\subseteq 	\sigma_r\left(B(r,s),l_2\left(\frac{1}{v_k}\right)\right)\subseteq \overline{\D}_{e^{-kl}|s|}(r).
	\end{equation}
	So, by Lemma \ref{L33}(i),(iii), we immediately obtain that 
\[
\sigma_p(B(r,s),\Lambda_\infty(\alpha)')=\emptyset
\]
and that
\[
 \sigma(B(r,s),\Lambda_\infty(\alpha)')\subseteq\cap_{m\in\N}\cup_{k\geq m} \overline{\D}_{e^{-kl}|s|}(r)=\cap_{m\in\N}\overline{\D}_{e^{-ml}|s|}(r)=\{r\}.
\]
On the other hand, $r\in \sigma_r(B(r,s),\Lambda_\infty(\alpha)')\subseteq \sigma(B(r,s),\Lambda_\infty(\alpha)')$. So, we can conclude that $\sigma_r(B(r,s),\Lambda_\infty(\alpha)')= \sigma(B(r,s),\Lambda_\infty(\alpha)')=\{r\}$ and that $\sigma_c(B(r,s),\Lambda_\infty(\alpha)')=\emptyset$. It remains to show that $\sigma^*(B(r,s),\Lambda_\infty(\alpha)')=\{r\}$. To this end, let $\lambda\in\C\setminus\{r\}$ be fixed. Then $|\lambda-r|>0$ and hence there exists $k_0\in\N$ such that  $|\lambda-r|>2e^{-k_0l}|s|$,  as $e^{-kl}\to 0$ for $k\to\infty$. It follows that  $\overline{\D}_{e^{-k_0l}|s|}(\lambda)\subseteq  \rho\left(B(r,s),l_2\left(\frac{1}{v_{k}}\right)\right)$ for every $k\geq k_0$.  On the other hand, $\overline{\D}_{e^{-k_0l}|s|}(\lambda)\subseteq \rho(B(r,s),\Lambda_\infty(\alpha)')$, and for every $\mu\in\overline{\D}_{e^{-k_0l}|s|}$ and $k\geq k_0$,
 the operator $R(\mu,B(r,s))$ when  restricted to $l_2\left(\frac{1}{v_{k}}\right)$ coincides with the resolvent operator of $B(r,s)$ acting on $l_2\left(\frac{1}{v_{k}}\right)$, i.e., with $(B(r,s)-\mu I)^{-1}\colon l_2\left(\frac{1}{v_{k}}\right)\to l_2\left(\frac{1}{v_{k}}\right)$. Consequently, the family $\{R(\mu,B(r,s)):\ \mu\in \overline{\D}_{e^{-k_0l}|s|}(\lambda)\}$ forms an equicontinuous set in $\cL(\Lambda_\infty(\alpha)')$.
Indeed,  for a  fixed $x\in \Lambda_\infty(\alpha)'$, there exists $k\geq k_0$ such that $x\in  l_2\left(\frac{1}{v_{k}}\right)$. By the equicontinuity  of  resolvent operators over bounded closed subsets of the resolvent set of a continuous operator acting in a Banach space, the set $\{R(\mu,B(r,s))x:\ \mu\in \overline{\D}_{e^{-k_0l}|s|}(\lambda)\}$ is  bounded in $l_2\left(\frac{1}{v_{k}}\right)$ and hence also in $\Lambda_\infty(\alpha)'$. From the arbitrariness of $x$ and the fact that $\Lambda_\infty(\alpha)'$ is a barrelled space, it follows that $\{R(\mu,B(r,s)):\ \mu\in \overline{\D}_{e^{-k_0l}|s|}(\lambda)\}$ is an equicontinuous subset of $\cL(\Lambda_\infty(\alpha)')$.
\end{proof} 

We conclude this section by studying the ergodic properties of the operator
 $B(r,s)$ acting on the spaces $\Lambda_\infty(\alpha)$ and $\Lambda_\infty(\alpha)'$. 
 To this end, we first recall the following results, originally proved in \cite[Theorems 24 and 26]{AA}, now stated in the framework of Fr\'echet and (LB) spaces.

\begin{lem}\label{L34}
Let $X =\cap_{k\in\N}X_k$ be a Fr\'echet space, where $X_k$ is a Banach space such that the inclusion $X_{k+1}\hookrightarrow X_{k}$ is continuous, for every $k\in\N$. Let $T\in \cL(X)$ satisfy condition (A) of Lemma \ref{L32}. Then the following properties are satisfied.
\begin{itemize}
\item[\rm (i)] If $T_k:= T_{|X_k}$ is power bounded in $X_k$ for every $k\in\N$, then $T$ is power bounded in $X$.
\item[\rm(ii)] If $T_k$ is mean ergodic in $X_k$ for every $k\in\N$, then $T$ is mean ergodic in $X$.
\item[\rm (iii)] If $T_k$ is uniformly mean ergodic in $X_k$ for every $k\in\N$, then $T$ is mean ergodic in $X$.
\item[\rm (iv)] If $(T_k)^n\to 0$ in $\cL_b(X_k)$  as $n\to\infty$ for every $k\in\N$, then $T^n\to 0$  in $\cL_b(X)$ as $n\to\infty$.
\end{itemize}
\end{lem}

We recall that an  (LB)-space $X = \cup_{k\in\N}  X_k$ is said to be \textit{regular} if every bounded subset $B$ of $X$ is contained
and bounded in $X_k$, for some $k\in\N$. Every complete (LB)-space is regular.

\begin{lem}\label{L35}
 Let $X = \cup_{k\in\N}  X_k$ be an Hausdorff inductive limit of a sequence of Banach spaces 
such that the inclusion $X_k\hookrightarrow X_{k+1}$ is continuous, for every $k\in\N$ (i.e., an (LB)-space) and let $T\in \cL(X)$ satisfy condition (B) of Lemma \ref{L33}. Then the following properties are satisfied.
\begin{itemize}
\item[\rm (i)] If $T_k:= T_{|X_k}$ is power bounded in $X_k$ for every $k\in\N$, then $T$ is power bounded in $X$.
\item[\rm (ii)] If $T_k$ is mean ergodic in $X_k$ for every $k\in\N$, then $T$ is mean ergodic in $X$.
\item[\rm (iii)] If $T_k$ is uniformly mean ergodic in $X_k$ for every $k\in\N$ and $X$ is regular, then $T$ is uniformly mean ergodic in $X$.
\item[\rm (iv)] If $(T_k)^n\to 0$ in $\cL_b(X_k)$  as $n\to\infty$ for every $k\in\N$ and $X$ is regular, then $T^n\to 0$  in $\cL_b(X)$ as $n\to\infty$.
\end{itemize}
\end{lem}

\begin{rem}  Statement (iv) in Lemmas \ref{L34} and \ref{L35} is not explicitly included in the cited results, but the proof proceeds analogously by arguing as in the case of the statement (iii) in \cite[Theorems 24 and 26]{AA}. 
\end{rem}

Taking into account that $\Lambda_\infty(\alpha)$ is a Fr\'echet space and that $\Lambda_\infty(\alpha)'$ is a complete (LB)-space, we can give the following results.

\begin{prop}\label{P.UM_0} Let $\alpha=\{\alpha_n\}_{n\in\N_0}\subseteq \R$ be a non-negative increasing sequence satisfying  the conditions $\lim_{n\to\infty}\alpha_n=\infty$,  $\lim_{n\to\infty}(\alpha_{n+1}-\alpha_n)=0$,  and $\alpha_n\geq c\log(n+1)$ for every $n\in\N_0$ and some $c>0$.
	Let $r,s\in\R$ and $r,s\neq0$.
	
{\rm (a)}	Consider the following statements:
\begin{enumerate}
\item $B(r,s)$ is mean ergodic  in $\cL(\Lambda_\infty(\alpha))$; 
\item $\frac{B^n(r,s)}{n}\to0$ in $\cL_s(\Lambda_\infty(\alpha))$ as $n\to\infty$;
\item $|r|+|s|\leq 1$.
\end{enumerate}
Then $(1)\Rightarrow(2)\Rightarrow(3)$.

{\rm (b)} Moreover, consider the following statements:
\begin{enumerate}
\item $|r|+|s|< 1$;
\item $B^n(r,s)\to0$ in $\cL_b(\Lambda_\infty(\alpha))$  as $n\to\infty$;
\item $B(r,s)$ is power bounded in $\cL(\Lambda_\infty(\alpha))$;
\item $B(r,s)$ is uniformly mean ergodic in $\cL(\Lambda_\infty(\alpha))$;
\item $1\not\in\sigma(B(r,s),\Lambda_\infty(\alpha))$.
\end{enumerate}
Then $(1)\Rightarrow(2)\Rightarrow(3)\Rightarrow(4)\Rightarrow(5)$.

{\rm (c)} Finally,   $B(r,s)$ is not supercyclic.
\end{prop}

\begin{proof} We first consider the case (a).
	
(1)$\Rightarrow$(2): It is straightforward.

(2)$\Rightarrow$(3): If $\frac{B^n(r,s)}{n}\to0$ in $\cL_s(\Lambda_\infty(\alpha))$ as $n\to\infty$, by \cite[Proposition 5.1]{AM} it follows  that $\sigma(B(r , s ), \Lambda_\infty(\alpha))\subseteq \overline{\D}$. In view of Theorem \ref{T.s}(1) and Lemma \ref{incspe}, we get that $|r|+|s|\leq 1$. 

We now pass to the case (b).

(1)$\Rightarrow$(2): We first recall that $\lim_{n\to\infty}\frac{v_k(n+1)}{v_k(n)}=1$ for every $k\in\N$. Therefore, by Proposition \ref{meulpv} it follows for each $k\in\N$  that $B^n(r,s)\to 0$ in $\cL_b(l_2(v_k))$ as $n\to\infty$. Since $\Lambda_\infty(\alpha)=\cap_{k\in\N}l_2(v_k)$ is the intersection of the sequence $(l_2(v_k))_{k\in\N}$ of Banach spaces and the assumption (A) in Lemma \ref{L32} is clearly satisfied, we can apply Lemma \ref{L34}(iv) to conclude that $B^n(r,s)\to 0$ in $\cL_b(\Lambda_\infty(\alpha))$.

(2)$\Rightarrow$(3): The assumption implies that $\sup_{n\in\N}|B^n(r,s)x|_{k}<\infty$ for every $x\in \Lambda_\infty(\alpha)$ and $k\in\N$. So, an application of  Banach-Steinhaus theorem yields the result.

(3)$\Rightarrow$(4): Since $\Lambda_\infty(\alpha)$ is a  Fr\'echet Schwartz space and hence a  Fr\'echet Montel space, the assumption implies that $B(r,s)$ is uniformly mean ergodic (see \cite[Proposition 2.8]{ABR-0}).

(4)$\Rightarrow$(5): Since $B(r,s)$ is uniformly mean ergodic, $B(r,s)$ is also mean ergodic and hence, we have
\begin{equation}\label{eq.DF}
	\overline{(I-B(r,s))(\Lambda_\infty(\alpha))}\oplus \Ker (I-B(r,s))=\Lambda_\infty(\alpha)
\end{equation}
see \cite[Ch. VIII; \S 3]{Y}. Since $\sigma_p(B(r,s),\Lambda_\infty(\alpha))=\emptyset$ (cf. \eqref{eq.Spcs}), we have that $\Ker (I-B(r,s))=\{0\}$. In view of \eqref{eq.DF} it follows that $	\overline{(I-B(r,s))(\Lambda_\infty(\alpha))}=\Lambda_\infty(\alpha)$.  On the other hand, also $\sigma_c(B(r,s),\Lambda_\infty(\alpha))=\emptyset$ (cf. \eqref{eq.Spcs}). This means that the operator $(I-B(r,s))\colon \Lambda_\infty(\alpha)\to \Lambda_\infty(\alpha)$ is bijective and hence, $1\in \rho(B(r,s),\Lambda_\infty(\alpha))$.

(c) Since  $\sigma_p(B(r,s)',\Lambda_\infty(\alpha)')
=\sigma_r(B(r,s),\Lambda_\infty(\alpha))=\overline{\D}_{|s|}(r)$, we can apply  \cite[Proposition 1.26]{BM} to conclude that the operator $B(r,s)$ cannot be supercyclic in $\Lambda_\infty(\alpha)$.\end{proof}

\begin{prop}\label{P.UM_l} Let $\alpha=\{\alpha_n\}_{n\in\N_0}\subseteq \R$ be a non-negative increasing sequence satisfying  the conditions $\lim_{n\to\infty}\alpha_n=\infty$ and   $\lim_{n\to\infty}(\alpha_{n+1}-\alpha_n)=:l\in (0,\infty)$. The for each $r,s\in\R$ and $r,s\neq0$ the operator $B(r,s)$  is not power bounded neither (uniformly) mean
	ergodic and fails to be supercyclic.
	\end{prop}

	\begin{proof}
		Observe that $\frac{B(r,s)^n}{n}\not\to 0$ in $\cL_s(\Lambda_\infty(\alpha))$ for $n\to\infty$. Otherwise, by \cite[Proposition 5.1]{AM}  we necessarily have $\sigma (B(r,s),\Lambda_\infty(\alpha))\subseteq \overline{\D}$. But this is not the
		case. Indeed, by Theorem \ref{T.s}(2) we have that $\sigma (B(r,s),\Lambda_\infty(\alpha)) = \C$. 
		Since the condition $\frac{B(r,s)^n}{n}\not\to 0$ in $\cL_s(\Lambda_\infty(\alpha))$ for $n\to\infty$ is necessary both for the power
		boundedness and for the mean ergodicity, we can conclude that the operator $B(r,s)$ is
		not power bounded neither mean ergodic.
		
		By  Theorem \ref{T.s}(2)  we also have that $\sigma_p(B(r,s)',\Lambda_\infty(\alpha)')
		=\sigma_r(B(r,s),\Lambda_\infty(\alpha))=\C$. So, we can apply  \cite[Proposition 1.26]{BM} to conclude that the operator $B(r,s)$ cannot be supercyclic in $\Lambda_\infty(\alpha)$.
	\end{proof}

\begin{prop} Let $\alpha=\{\alpha_n\}_{n\in\N_0}\subseteq \R$ be a non-negative increasing sequence satisfying  the conditions $\lim_{n\to\infty}\alpha_n=\infty$ and  $\lim_{n\to\infty}(\alpha_{n+1}-\alpha_n)=:l\in [0,\infty)$.
Let $r,s\in\R$ and $r,s\neq0$.

{\rm (a)} Consider the following statements:
\begin{enumerate}
\item $B(r,s)$ is power bounded in $\cL(\Lambda_\infty(\alpha)')$;
\item $B(r,s)$ is uniformly mean ergodic  in $\cL(\Lambda_\infty(\alpha)')$; 
\item $B(r,s)$ is mean ergodic  in $\cL(\Lambda_\infty(\alpha)')$; 
\item $\frac{B^n(r,s)}{n}\to0$ in $\cL_s(\Lambda_\infty(\alpha)')$ as $n\to\infty$;
\item $|r|+|s|\leq 1$ for $l=0$ and $|r|\leq 1$ for $l>0$.
\end{enumerate}
Then for $l=0$ all the conditions $(1)\div(5)$ are equivalent and for $l>0$, (1)$\Rightarrow$(2)$\Rightarrow$(3)$\Rightarrow$(4)$\Rightarrow$(5).    
{\rm (b)} Moreover, consider the following statements:
\begin{enumerate}
\item $|r|+|s|< 1$;
\item $B^n(r,s)\to0$ in $\cL_b(\Lambda_\infty(\alpha)')$  as $n\to\infty$.
\end{enumerate}
Then $(1)\Rightarrow(2)$.

{\rm (c)} Let $l=0$. Then $B(r,s)$ is not supercyclic.
\end{prop}

\begin{proof} We first consider the case (a).
	
(1)$\Rightarrow$(2): Since $\Lambda_\infty(\alpha)'$ is  a Montel (LB)-space, the assumption implies that $B(r,s)$ is uniformly mean ergodic (see \cite[Proposition 2.8]{ABR-0}). 

(2)$\Rightarrow$(3): It is straightforward.

(3)$\Rightarrow$(4): It is straightforward.

(4)$\Rightarrow$(5): We first observe that $\Lambda_\infty(\alpha)'$ is a sequentially complete barrelled lcHs because it is a complete (LB)-space. So, if $\frac{B^n(r,s)}{n}\to0$ in $\cL_s(\Lambda_\infty(\alpha)')$ for $n\to\infty$, we can apply  \cite[Remark 5.3]{AM} to conclude that $\sigma(B(r , s ), \Lambda_\infty(\alpha)')\subseteq \overline{\D}$. In view of Theorem \ref{T.s1}(1)-(2) and Lemma \ref{incspe}, it follows that $|r|+|s|\leq 1$ for $l=0$ and $|r|\leq 1$ for $l>0$. 

(5)$\Rightarrow$(1): Let $l=0$. Since for each $k\in\N$,  $\sup_{n\in\N_0}\frac{\frac{1}{v_k(n+1)}}{\frac{1}{v_k(n)}}=e^{-k\inf_{n\in\N_0}(\alpha_{n+1}-\alpha_n)}=:N'_k\leq 1$,   the condition  $ |r|+|s|\leq 1$ implies  that $|r|+N'_k|s|\leq 1$ for every $k\in\N$. By  Proposition \ref{melpv} it follows that the operator $B(r,s)$ is power bounded in $l_2\left(\frac{1}{v_k}\right)$ for every $k\in\N$. Taking into account that $\Lambda_\infty(\alpha)'=\cup_{k\in\N} l_2\left(\frac{1}{v_k}\right)$ is the inductive limit of the sequence $\left(l_2\left(\frac{1}{v_k}\right)\right)_{k\in\N}$ of Banach spaces and the condition (B) in Lemma \ref{L33} is clearly satisfied, we can apply  Lemma \ref{L35}(i) to conclude that $B(r,s)$ is power bounded.

We now pass to the case (b).

(1)$\Rightarrow$(2): For each $k\in\N$ we have that $\lim_{n\to\infty}\frac{\frac{1}{v_k(n+1)}}{\frac{1}{v_k(n)}}=e^{-kl}\leq 1$. So,  the condition  $|r|+|s|< 1$ implies that  $|r|+e^{-kl}|s|< 1$ for every $k\in\N$. By  Proposition \ref{meulpv} it follows that 
 $B^n(r,s)\to 0$  in $\cL_b\left(l_2\left(\frac{1}{v_k}\right)\right)$ as $n\to\infty$ for every $k\in\N$. Taking into account that $\Lambda_\infty(\alpha)'=\cup_{k\in\N} l_2\left(\frac{1}{v_k}\right)$ is the inductive limit of the sequence $\left(l_2\left(\frac{1}{v_k}\right)\right)_{k\in\N}$ of Banach spaces, in particular a regular (LB)-space,  and the condition (B) in Lemma \ref{L33} is clearly satisfied, we can apply  Lemma \ref{L35}(iv) to conclude that $B^n(r,s)\to 0$ in $\cL_b(\Lambda_\infty(\alpha)')$ as $n\to\infty$.

(c) For $l=0$ we have  that $\sigma_p(B(r,s)',\Lambda_\infty(\alpha))
=\sigma_r(B(r,s),\Lambda_\infty(\alpha)')={\D}_{|s|}(r)$ (cf. \eqref{eq.SprcSS}). So, we can apply  \cite[Proposition 1.26]{BM} to conclude that the operator $B(r,s)$ cannot be supercyclic.
\end{proof}

\medskip

{\bf Funding}  This work was  partially supported by the INdAM-GNAMPA.

\medskip

{\bf Author Contributions} All the authors wrote the manuscript text and  reviewed  it.

\medskip

{\bf Data availability} Not applicable.

\end{document}